\documentclass[11pt]{article}

\usepackage{amsfonts,amsmath,latexsym,color,epsfig,enumitem, amssymb,bbm,amsthm}
\usepackage{hyperref}
\usepackage{subcaption}
\usepackage{ulem}

\newtheorem{conj}{Conjecture}

\newtheorem{question}[conj]{Question}
\newtheorem{theorem}{Theorem}[section]

\newtheorem{proposition}[theorem]{Proposition}
\newtheorem{lemma}{Lemma}[section]

\theoremstyle{definition}

\newtheorem{remark}[lemma]{Remark}

\numberwithin{equation}{section}

\renewcommand{\le}{\leqslant}
\renewcommand{\ge}{\geqslant}
\renewcommand{\leq}{\leqslant}
\renewcommand{\geq}{\geqslant}
\newcommand{\R}{\mathbb R}

\newcommand{\mc}{\mathcal}
\newcommand{\dist}{\mathrm{dist}}

\DeclareMathOperator{\Area}{Area}

\def\qed{\ifvmode\mbox{ }\else\unskip\fi\hskip 1em plus 10fill$\Box$}

\input{epsf}

\makeatletter
\def\Ddots{\mathinner{\mkern1mu\raise\p@
\vbox{\kern7\p@\hbox{.}}\mkern2mu
\raise4\p@\hbox{.}\mkern2mu\raise7\p@\hbox{.}\mkern1mu}}
\makeatother

\def\R{\mathbb R}
\def\Z{\mathbb Z}

\def\T{\mathbb T}
\def\d{\mathrm d}

\DeclareMathOperator{\Err}{Err}
\DeclareMathOperator{\supp}{supp}

\title{\vspace{-0.7cm}A new upper bound for the Heilbronn triangle problem}
\author{Alex Cohen\thanks{Department of Mathematics, Massachusetts Institute of Technology, Cambridge, MA. Email: {\tt alexcoh@mit.edu}. Research supported by an NSF GRFP Fellowship and a Hertz Foundation fellowship.} \and Cosmin Pohoata\thanks{School of Mathematics, Institute for Advanced Study, Princeton, NJ. Email: {\tt pohoata@ias.edu}. Research supported by NSF Awards DMS-1926686 and DMS-2246659.} \and Dmitrii Zakharov\thanks{Department of Mathematics, Massachusetts Institute of Technology, Cambridge, MA. Email: {\tt  zakhdm@mit.edu}.}}
\date{}

\begin{document}
\maketitle

\begin{abstract}
\smallskip
For sufficiently large $n$, we show that in every configuration of $n$ points chosen inside the unit square there exists a triangle of area less than $n^{-8/7-1/2000}$. This improves upon a result of Koml\'os, Pintz and Szemer\'edi from 1982. Our approach establishes new connections between the Heilbronn triangle problem and various themes in incidence geometry and projection theory which are closely related to the discretized sum-product phenomenon.

\end{abstract}

\section{Introduction}
Given an integer $n \geq 3$, the Heilbronn triangle problem asks for the smallest number $\Delta = \Delta(n)$ such that in every configuration of $n$ points in the unit square $[0,1]^2$ one can always find three among them which form a triangle of area at most $\Delta$. A trivial upper bound of the form $\Delta = O(1/n)$ follows from the simple observation that one can always triangulate a set of $n$ points in $[0,1]^2$ and obtain at least $n -2$ triangles whose interiors are disjoint and are all contained in the unit square. By the pigeonhole principle, at least one of these triangles will have area at most $\frac{1}{n-2}$. There is also a simple lower bound $\Delta = \Omega(1/n^2)$, which follows from either a direct greedy construction, a probabilistic construction, or the following explicit algebraic construction due to Erd\H{o}s (see for example \cite{AS92} and \cite{Roth4}). For a prime $n \leq p \leq 2n$, define 
$$X = \left\{\left(\frac{x}{p},\frac{y}{p}\right): x, y \in \left\{0,\ldots,p-1\right\},\ y = x^2 \hspace{-3mm} \mod p\right\}.$$
This is a set of size $p$ inside $[0,1]^2$ with no three collinear points and, since all the elements in the dilated set $pX$ are lattice points, any triangle with vertices in $X$ must have area at least $1/2p^2 \geq 1/8n^2$. 

The question is named after Hans Heilbronn, who in the late 1940's originally conjectured that $\Delta = O\left(n^{-2}\right)$. Koml\'os, Pintz and Szemer\'edi \cite{KPS82} disproved this conjecture by showing the existence of a configuration of $n$ points in the unit square with all triangles of area $\geq c(\log n)/n^{2}$, for some constant $c> 0$. Nevertheless, the problem of finding strong asymptotic upper bounds for $\Delta$ has remained a longstanding open problem over the years. The first nontrivial upper bound for Heilbronn's triangle problem was obtained by Roth in 1951, who showed that $\Delta = o(1/n)$ must always hold. Roth's proof consisted of a remarkable density increment argument that in fact derived the following quantitative improvement over the trivial bound: $\Delta = O\left(n^{-1} (\log \log n)^{-1/2}\right)$. In many ways, this result can be also seen as a precursor of Roth's celebrated theorem about the maximum size of a set in $\left\{1,\ldots,n\right\}$ without nontrivial three-term arithmetic progressions \cite{Roth3AP}. The first quantitative improvement of Roth's upper bound came around in 1971 from Schmidt \cite{Schmidt}, who proved that $\Delta = O\left(n^{-1} (\log n)^{-1/2}\right)$ using a much simpler argument. In the subsequent year already, however, Roth returned and introduced a novel analytic method which established the first polynomial improvement over the trivial bound, showing that $\Delta \lessapprox n^{-1-\mu}$ must hold for an absolute constant $\mu > 0$. In fact, Roth (see \cite{Roth2}, and somewhat later \cite{Roth3}) obtained the following explicit values:
$$\mu = 1 - \sqrt{\frac{4}{5}} \approx 0.10557\ \ \ \text{and}\ \ \ \mu = \frac{1}{8}(9 - \sqrt{65})  \approx 0.11721.$$
Optimizing Roth's method, in 1981 Koml\'os, Pintz and Szemer\'edi ultimately managed to improve this exponent to $\mu = 1/7 \approx 0.14285$. More precisely, they showed that
\begin{equation} \label{KPS}
\Delta \leq \frac{\exp(c\sqrt{\log n})}{n^{8/7}}
\end{equation}
must hold for some absolute constant $c>0$. 

In this paper, we break this barrier for the Heilbronn triangle problem by establishing some new connections with various themes in incidence geometry and projection theory which are closely related to the discretized sum-product phenomenon. Our main result is the following:

\begin{theorem} \label{thm:main_theorem} 
For sufficiently large $n$, any set of $n$ points in the unit square contains three points forming a triangle of area at most $$\Delta \leq n^{-8/7-1/2000}.$$
\end{theorem}

A key ingredient in our argument is a recent theorem of Orponen, Shmerkin, and Wang from \cite{OSW}, so the polynomial improvement from Theorem \ref{thm:main_theorem} can be regarded in some sense as a manifestation of the discretized sum-product theorem of Bourgain \cite{Bourgain03}.
We will start discussing our approach in \S\ref{sec:Incidences}, along with an account of the important early ideas of Roth and Schmidt that our new framework builds on. As an intermediate step towards Theorem \ref{thm:main_theorem}, we will also give an alternative proof of the Koml\'os, Pintz and Szemer\'edi result.

The proof of Theorem \ref{thm:main_theorem} shows that there exists an absolute constant $c > 0$ such that $\Delta \leq n^{-8/7 - c}$, for some explicit constant $c$ arising from a system of inequalities. 
We made no serious attempts to optimize the exponent $c$ satisfying these inequalities, nor the system itself. It seems plausible that with a more careful analysis of the argument one can obtain a slightly better value for $c$ than $1/2000$. 

Our methods also allow us to give a simple proof of a further improvement in the case when $P$ is a {\it{homogeneous}} set of points. We say $P \subset [0,1]^2$ is homogeneous if there exists an absolute constant $C > 0$ such that in every axis-parallel $n^{-1/2} \times n^{-1/2}$ square $Q$ contains at most $C$ points from $P$. 

\begin{theorem} \label{separated}
For sufficiently large $n$, any homogeneous set of $n$ points in the unit square contains three points forming a triangle of area at most $$\Delta \leq n^{-7/6+\varepsilon}$$
for any $\varepsilon > 0$. 
\end{theorem}

\bigskip

{\bf{Notational conventions}}.
We will use the notation $f \lesssim g$ to denote that $f = O(g)$, that is, there exists some constant $C > 0$ such that $f < Cg$. We write $f \sim g$ if both $f \lesssim g$ and $g \lesssim f$ hold. Furthermore, whenever $f$ and $g$ are functions of the same variable $n$ we write $f \lessapprox g$ when $f < e^{c\sqrt{\log n}} g$ holds for some constant $c > 0$ which is independent of $n$, and $f \lll g$ when $f$ is asymptotically much smaller than $g$. We use $C$ to denote a numerical constant which may change from line to line.

\section{Incidence geometry setup}\label{sec:Incidences}

Let $P \subset [0,1]^2$ be a set of $n$ points and let ${P \choose 2}$ denote the set of non-degenerate pairs of points in $P$. For every $w > 0$ and $\tau = \{x,y\} \in {P \choose 2}$, let $\T_{\tau}(w)$ denote the strip of width $w$ containing all the points at distance $<w/2$ from the line $\ell_{\tau}$ supporting the pair $\tau$. Whenever more convenient, we shall say that $\T_{\tau}(w)$ is the (open) strip of width $w$ generated by $\tau$. The Heilbronn triangle problem is implicitly a problem involving incidences between points in $\mathbb{R}^{2}$ and strips generated by elements of ${P \choose 2}$. For $\Delta > 0$, if $d(\tau)$ denotes the length of the segment determined by $x$ and $y$, the locus of points $z$ such that $x,y,z$ forms a triangle with area $< \Delta$ is precisely the strip $\T_{\tau}\left(\frac{4\Delta}{d(\tau)}\right)$. In particular, if $\Delta$ denotes the smallest area determined by a triangle with vertices in $P$, then the strip $\T_{\tau}\left(\frac{4\Delta}{d(\tau)}\right)$ cannot contain any point $z$ from $P \setminus \left\{x,y\right\}$, or equivalently 
\vspace{-2mm}
\begin{equation} \label{emptystrips}
    \T_{\tau}\left(\frac{4\Delta}{d(\tau)}\right) \cap P = \left\{x,y\right\}\ \ \text{holds for every}\ \ \tau \in {P \choose 2}.
\end{equation}
Starting with Roth's original approach \cite{Roth1}, this simple observation has been the driving force behind all the progress on quantitative bounds for Heilbronn's triangle problem. The point is that if $\Delta$ is large, then these strips can be rather large, and there is also an unexpectedly small number of points from $P$ in each of them.

In \cite{Roth1}, Roth restricts to a subfamily $\mathcal{F} \subset {P \choose 2}$ of pairs $\tau$ where the corresponding segments have small lengths and the slopes of the supporting lines are similar, noticing that the strips $\left\{\T_{\tau}\left(\frac{4\Delta}{d(\tau)}\right) :\ \tau \in \mathcal{F}\right\}$ are essentially pairwise disjoint. Passing to the complement of the union of these strips, one can then find a region of the square where $P$ has an increased density (relative to its area), and subsequently iterate. 
Schmidt's argument \cite{Schmidt}, on the other hand, does not directly address this tension, but it is the first approach that implicitly considers an estimation of the number of edges in an incidence graph. Schmidt establishes the following relation between $\Delta$ and the so-called Riesz $2$-energy of $P$ (see for instance \cite{HS04}):
\begin{equation} \label{Riesz}
    \Delta \lesssim \left(\sum_{\tau \in {P \choose 2}} \frac{1}{d(\tau)^{2}}\right)^{-1/2}.
\end{equation}
The idea is as follows. Let $\mathcal{S} = [0,1]^2$ and consider again the full set of strips $L:=\left\{\T_{\tau}\left(\frac{4\Delta}{d(\tau)}\right): \tau \in {P \choose 2}\right\}$. For every $(x,\T) \in \mathcal{S} \times L$ where $\T = \T_{\tau}\left(\frac{4\Delta}{d(\tau)}\right)$, define the weight of the corresponding edge of the incidence graph of $\mathcal{S}$ and $L$ by $w(\T):=4\Delta/d(\tau)$, the width of the strip $\T$. Schmidt then noticed that the (weighted) degree of $x$ becomes bounded by a positive constant, which is independent of $n$. Equivalently, for every $x \in \mathcal{S}$ we have
\begin{equation} \label{bounded_deg}
    \sum_{\T \in L} w(\T)1_{\T}(x)  \lesssim 1,
\end{equation}
where $1_{\T}$ denotes the usual characteristic function of $\T$. This in turn implies that the total number of weighted incidences $I_{w}(\mathcal{S},L)$ determined by $\mathcal{S}$ and $L$ satisfies
$$I_{w}(\mathcal{S},L) = \int_{x \in \mathcal{S}} \sum_{\T \in L} w(\T)1_{\T}(x) \d x \lesssim 1.$$
On the other hand, 
$$\int_{x \in \mathcal{S}} \sum_{\T \in L} w(\T)1_{\T}(x) \d x = \int_{x \in \mathcal{S}} \sum_{\tau \in {P \choose 2} }  \frac{4\Delta}{d(\tau)} 1_{\T_{\tau}\left(\frac{4\Delta}{d(\tau)}\right)}(x) \d x = \sum_{\tau \in {P \choose 2} } \frac{4\Delta}{d(\tau)} \int_{x \in \mathcal{S}} 1_{\T_{\tau}\left(\frac{4\Delta}{d(\tau)}\right)}(x) \d x$$
and
$$\int_{x \in \mathcal{S}} 1_{\T_{\tau}\left(\frac{4\Delta}{d(\tau)}\right)}(x) \d x \gtrsim \frac{\Delta}{d(\tau)},$$
so (\ref{Riesz}) follows. 

For any set $P$, there is a lower bound $\sum_{\tau \in {P \choose 2}} \frac{1}{d(\tau)^{2}} \gtrsim n^{2} \log n$.
There are $\sim \log n$ many distance scales, and every distance scale contributes $\gtrsim n^2$ to the Riesz energy.
It follows that
$$\Delta^{2} \cdot n^{2} \log n \lesssim \Delta^{2} \cdot \sum_{\tau \in {P \choose 2}} \frac{1}{d(\tau)^{2}} \lesssim 1,$$
and so $\Delta \lesssim n^{-1} (\log n)^{-1/2}$ holds for every set of $n$ points $P \subset [0,1]^2$. See \cite{Schmidt} for more details. From this perspective, it becomes natural to wonder whether one can use incidence geometry in order to make progress on the Heilbronn triangle problem. 

In 1972, Roth \cite{Roth2} (indirectly) managed to do so and provided the first polynomial improvement over the trivial upper bound by introducing an elegant analytic method to better exploit the sparsity of the strips from \eqref{emptystrips}. In modern language, the idea was as follows. For every $u > n^{-1/2}$ consider the pairs $\tau = \{x,y\} \in {P \choose 2}$ such that $d(\tau) \leq u$. For each such $\tau$ let $\ell_{\tau}$ denote the line supporting $\tau$, let $L = \{\ell_{\tau}\, :\, d(\tau) \leq u\}$, and let $d(P, L)$ denote the smallest nontrivial distance between a point of $P$ and line of $L$. First, notice that
$$\Delta \leq u\, d(P, L)$$
holds for every choice of $u$. 
The number of incidences between $P$ and $L$ at scale $w > 0$ is given by $\#\{(p, \ell) \in P\times L\, :\, p \in \T_{\ell}(w)\}$. Here $\ell=\ell_{\tau}$ and $\T_\ell(w)$ is the strip of width $w$ generated by $\tau$. 
It will be convenient to work with a smoothed out version of incidences. 
In \S\ref{subsec:high_low_proof} we will construct a certain bump function $\eta: \R \to [0,1]$ satisfying
\begin{equation*}
\eta|_{[-1/2+1/10,1/2-1/10]} = 1\quad\text{ and }\quad \supp \eta \subset [-1/2-1/10, 1/2+1/10].  
\end{equation*}
We define smoothed incidences by 
\begin{equation}\label{eq:defn_smoothed_incidences_formula}
    I(w; P, L) = \sum_{(p, \ell) \in P \times L} \eta(w^{-1}d(p, \ell)). 
\end{equation}

Now, if one can show that $I(w_f; P, L) \gtrsim w_f |P| |L|$ holds for some final scale $w_f$, then $\Delta \lesssim u w_f$. For instance, one could choose $u = n^{-1/3}$, $w_f = n^{-3/4}$, $\Delta = n^{-13/12}$. With this in mind, Roth's method has two main steps. Given a set of points $P \subset [0,1]^2$ with no small area triangles:

\begin{itemize}
 	\item[(A)] \textbf{Initial estimate.} Show that for some initial scale $w_i$,
 	\begin{equation}\label{eq:roth_initial_estimate}
 		I(w_i; P, L) \gtrsim w_i |P||L|.
 	\end{equation}

 	\item[(B)] \textbf{Inductive step.} Show that at some final scale $w_f \lll w_i$ there are still lots of incidences,
 	\begin{equation}\label{eq:roth_inductive_estimate}
 		\left|\frac{I(w_i; P, L)}{w_i |P||L|} - \frac{I(w_f; P, L)}{w_f |P||L|}\right| \lll 1.
 	\end{equation}
\end{itemize} 

Notice that for every $w > 0$, $w|P||L|$ is the expected number of incidences between $P$ and the set of strips $\T_\ell(w)$ where $\ell \in L$, if the points from $P$ were distributed uniformly at random in $[0,1]^2$. So the first step pinpoints a scale $w_{i}$ at which every set $P \subset [0,1]^2$ that lacks small triangles exhibits (pseudo)random behavior. At the other extreme, the strips $\T_{\ell}\left(\frac{4\Delta}{u}\right)$ each contain an unexpectedly small number of points: if $\ell=\ell_{\tau}$, then $\T_{\ell}\left(\frac{4\Delta}{u}\right) \subset \T_{\ell}\left(\frac{4\Delta}{d(\tau)}\right)$ and recall that the latter strip satisfies \eqref{emptystrips}. 
As one varies the scale from $w_i$ (rich strips) to $\frac{4\Delta}{u}$ (empty strips) it is sensible to expect that a phase transition in behavior must happen. One can quantify this behavior analytically as follows.

Given any two strips $\T_{\ell_1}(w_f)$ and $\T_{\ell_2}(w_i)$, where $w_f < w_i$, the area of the parallelogram of intersection (if finite) is proportional to each of $w_f$ and $w_i$ and that enables one to construct systems of orthogonal functions from (weighted versions of) these strips: if for every line $\ell$, we let 
$$\phi_{\ell}(w_i,w_f;x) = \frac{1}{w_i} 1_{\T_{\ell}(w_i)}(x) - \frac{1}{w_f} 1_{\T_{\ell}(w_f)}(x),$$
then for any two lines $\ell_1, \ell_2 \in L$ we have that 
$$\int_{\mathbb{R}^{2}} \phi_{\ell_{1}}(w_i,w_f;x) \phi_{\ell_{2}}(w_i,w_f;x) \d x = 0$$
whenever it is finite, for any choice of $w_i$ and $w_f$. If $\psi_{\ell}$ denotes the function obtained from $\phi_{\ell}$ by defining $\psi_{\ell}$ to be $0$ outside the square $[-2,2]^2$, namely $\psi_{\ell}(w_i,w_f;x)=1_{[-2,2]^2}\phi_{\ell}(w_i,w_f;x)$, then an appropriate application of Selberg's inequality (a generalization of Bessel's inequality) with respect to the system of quasi-orthogonal functions $\left\{\psi_{\ell}\right\}_{\ell \in L}$ readily provides an inequality relating the number of incidences at any two different scales $w_f<w_i$. With this tool in hand, one can then start with $w_i$ and pinpoint a much smaller scale $w_f$ that satisfies \eqref{eq:roth_inductive_estimate}\footnote{In his original paper on this method \cite{Roth2}, Roth applied Selberg's inequality only once to compare incidences at an initial scale and final scale. His later refinement from \cite{Roth3} exploits the fact that it is beneficial to apply the inequality several times at a sequence of intermediate scales.}. We refer to the excellent survey of Roth \cite{Roth4} more details on this perspective. 

In this paper, we take a different (and morally equivalent) route, via the so-called {\it{high-low method}} introduced in 2017 by Guth, Solomon, and Wang \cite{GSW} to prove upper bounds for incidences determined by balls and well-spaced tubes. This allows one to derive the following comparison between the number of (renormalized) incidence counts at two different scales:
$$\left|\frac{I(w_i; P, L)}{w_i\, |P||L|} - \frac{I(w_f; P, L)}{w_f\, |P||L|}\right| \lesssim \left(\frac{M_P(w_f\times w_f)}{|P|} \frac{M_L(w_i\times 1)}{|L|} \,w_f^{-3}\right)^{1/2}.$$

Here $M_P(w_f\times w_f)$ denotes the maximum number of points in a $w_f\times w_f$ square $Q$, and $M_L(w_i \times 1)$ is the maximum number of lines whose intersection with $[0,1]^2$ is fully contained in a $w_i \times 1$ tube $T$. The inequality above thus roughly says: if the points and lines are not too concentrated, then the normalized number of incidences doesn't change too much as the scale varies. 

In \S\ref{sec:inductive_step_high_low} we prove this estimate and outline how it is used for the inductive step.

In this paper we will also establish a new connection between Heilbronn's triangle problem and projection theory, which is the main novelty in our approach. Projection theory is an area of geometric analysis that started with a fundamental paper of Marstrand \cite{Marstrand} from 1954, who showed that if $X$ is a Borel set in $\mathbb{R}^2$, then the projection of $A$ onto almost every line through the origin has Hausdorff dimension $\min\left\{1, \operatorname{dim}_{H}(X)\right\}$. See for example the survey of Falconer \cite{FalconerSurvey}. 

An elegant result that can be derived from \cite{Marstrand} is the following estimate on the Hausdorff dimension of the set of directions determined by a set in $\mathbb{R}^{2}$. 

\begin{theorem} \label{Mars}
Let $X \subset \R^2$ be a nonempty Borel set with Hausdorff dimension $\dim_H X > 1$. Then 
$$\dim_H S(X) = 1,$$
where $S(X) \subset S^1$ denotes the set of directions spanned by $X$, i.e. $S(X) := \Bigl\{\frac{x-y}{|x-y|}\ :\ x, y \in X\Bigr\}$.
\end{theorem}
See \cite{Orp} for a discussion and further references. 

The first important twist in our story is the fact that a discretized version of this result can be used to prove an initial estimate of the form \eqref{eq:roth_initial_estimate} which turns out to be similar in quality to Roth's original estimate. We will discuss this in \S\S\ref{sec:inital_estimate}-\ref{sec:lower_bounds_for_incidences}. We can then use this estimate to give a new proof of the Koml\'os, Pintz, and Szemer\'edi bound $\Delta \lessapprox  n^{-8/7}$. 

Estimating the dimension of the direction set determined by a set of points is closely related to the problem of estimating dimensions of radial projections, an active topic recently. See for example \cite{OSW} and the references therein. One of the main results of Orponen, Shmerkin, and Wang from \cite{OSW} is the following remarkable refinement of Theorem \ref{Mars}.

\begin{theorem}\label{thm:OSW_dir_set}
Let $X \subset \R^2$ be a nonempty Borel set not contained in any line. Then
$$
\vspace{+3mm}
	\dim_H S(X) \geq \min\left\{1, \dim_H X\right\}. 
$$
\end{theorem}

Theorem \ref{thm:OSW_dir_set} can also be regarded as a continuous analogue of Sz\H{o}ny's theorem \cite{Szony} that every set $A \subset \mathbb{F}_{p}^{2}$ of size $1 < |A| \leq p$ determines at least $\frac{|A|+3}{2}$ distinct directions, provided that $A$ is not contained in any affine line. This result extends a previous theorem of R\'edei \cite{Redei} which addressed the case when $|A|=p$, and which corresponds to Theorem \ref{Mars} in this analogy. 

Using a discretized version of Theorem \ref{thm:OSW_dir_set} rather than of Theorem \ref{Mars}, we will be able to access finer scales and thus manage to provide an initial estimate in the style of \eqref{eq:roth_initial_estimate} that allows us to bypass the Koml\'os, Pintz, and Szemer\'edi threshold.

\section{Roth's analytic method: a modern take}\label{sec:inductive_step_high_low}

Now we state the high-low inequality which drives the inductive step. Let $w > 0$ be a scale. For $T$ a $w$-tube, let
\begin{equation*}
    L\cap T = \{\ell \in L\, :\, \ell \cap [0,1]^2 \subset T\}
\end{equation*}
be the set of lines contained in $T$. We sometimes abbreviate $\ell \cap [0,1]^2 \subset T$ by just saying $\ell \subset T$. In general, when evaluating subset relations always intersect the sets involved with $[0,1]^2$ first. 
Define
\begin{align*}
	M_P(w\times w) &= \max_{\text{$Q$ a $w\times w$ square}} |P \cap Q|, \\
	M_L(w\times 1) &= \max_{\text{$T$ a $w$-tube}}|L \cap T|.
\end{align*}

Furthermore, let
\begin{equation*}
    B(w; P, L) = \frac{I(w; P, L)}{w\, |P||L|}
\end{equation*}
denote the normalized incidence count so that if $P$ and $L$ are distributed uniformly then $B(w; P, L) \sim 1$. To simplify notation, we will sometimes write $B(w)$ instead of $B(w;P,L)$, and similarly with other functions.

\begin{theorem}[High-low bound, single scale]\label{thm:high_low_bound}
We have
\begin{equation}\label{eq:high_low_bound}
      |B(w; P, L) - B(w/10; P, L)| \lesssim \left(\frac{M_P(w\times w)}{|P|} \frac{M_L(w\times 1)}{|L|} \,w^{-3}\right)^{1/2}. 
\end{equation}
\end{theorem}

We can write the error term as 
\begin{equation}\label{eq:defn_err_equation}
      \Err(w; P, L) = \left(\frac{M_P(w\times w)}{|P|} \frac{M_L(w\times 1)}{|L|} \,w^{-3}\right)^{1/2}.
\end{equation}
The definition of smoothed incidences is specially designed to accommodate the factor of $10$ in (\ref{eq:high_low_bound}): for example, in order to replace $w/10$ with $w/8$, you would have to change the smoothing function $\eta$ in (\ref{eq:defn_smoothed_incidences_formula}). 
Now we state some variations on the bound (\ref{eq:high_low_bound}). For any $K > 1$, if  $10^j \leq K < 10^{j+1}$, then 
\begin{equation*}
    B(w/10^{j+1}) \lesssim B(w/K) \lesssim B(w/10^j).
\end{equation*}
Using the inequality $\Err(w/10^j) \leq 10^{3j/2}\Err(w)$, we can geometrically sum to obtain
\begin{equation}\label{eq:high_low_bound_K}
      B(w) - K^{3/2} \Err(w) \lesssim B(w/K) \lesssim B(w) + K^{3/2} \Err(w).
\end{equation}
To compare the number of incidences at two far apart scales $w_f < w_i$ (think of $w_i \sim |P|^{-0.1}$ and $w_f \sim |P|^{-2/3}$), we sum the errors at several intermediate scales to get the estimate
\begin{equation}\label{multscale}
    B(w_f) - \mathrm{Err} \lesssim B(w_f) \lesssim B(w_i) + \mathrm{Err},\quad \mathrm{Err} = \log\Bigl(2+\frac{w_i}{w_f}\Bigr) \sup_{w_f < w < w_i} \Err(w). 
\end{equation}

In order to run the inductive step we need $\Err(w) \lll 1$ for all $w_f < w < w_i$. The error is small if the point set and line set are not too concentrated. On one extreme, if the points and lines are as well spaced as possible, then $M_P(w\times w) \sim \max\left\{1, w^2 |P|\right\}$ and $M_L(w\times 1) \sim \max\left\{1, w^2 |L|\right\}$. Usually such a bound is too much to ask for. The Frostman regularity condition is a natural way of controlling how much points and lines are concentrated, it says that for some $0 < \alpha_P, \alpha_L \leq 2$, 
\begin{equation*}\label{eq:frostman_cond}
	M_P(w\times w) \lesssim \max\left\{1, w^{\alpha_P} |P|\right\},\quad M_L(w\times 1) \lesssim \max\left\{1, w^{\alpha_L} |L|\right\}.
\end{equation*}
This condition is a discrete analogue of $P$ being an $\alpha_P$-dimensional set and $L$ being an $\alpha_L$-dimensional set. 
In order to run the inductive step, we need $\alpha_P + \alpha_L > 3$. 

For small enough $w$, we expect $M_P(w\times w) \sim M_L(w\times 1) \sim 1$. In this case $\Err(w) = (|P||L|)^{-1/2}w^{-3/2}$, which is controlled if $w \gtrsim (|P||L|)^{-1/3}$. So the smallest length scale reachable by the high-low method is $w_f = (|P||L|)^{-1/3}$. There is an example behind this barrier. 

\bigskip

{\bf{Szemer\'edi-Trotter obstruction}}. The sharp example for the high-low inequality is the same as the sharp example for the Szemer\'edi-Trotter theorem \cite{ST83}. Let 
\begin{align*}
	P = \Bigl\{\Bigl(\frac{a}{\sqrt{n}}, \frac{b}{\sqrt{n}}\Bigr)\, :\, 0 \leq a, b \leq n^{1/2}\Bigr\} \subset [0,1]^2
\end{align*}
be the square grid. Let 
\begin{align*}
	L = \Bigr\{\ell_{p, \vec q}\, :\, p, q \in P,\ \|q\|_{\infty} \sim n^{-1/3}\},\quad \ell_{p,\vec q} \text{ is the line based at $p$ in direction $q$}. 
\end{align*}
Each line in $L$ passes through $\sim n^{1/3}$ points of $P$, and each point $p \in P$ has $\sim n^{1/3}$ lines through it, so $|P| \sim |L| = n$.
We have
\begin{align*}
    B(w; P, L) \sim \max\left\{1, w^{-1}n^{-2/3}\right\}.
\end{align*}
The two terms in the maximum agree when $w = n^{-2/3}$, which is the smallest distance between a line and a point off the line. 
We have
$$M_P(w\times w) \sim \max\left\{1, w^2 |P|\right\},\quad M_L(w\times 1) \sim \max\left\{1, w^2 |L|\right\},\ \text{and}\ \Err(w) \sim \max\left\{w^{1/2}, n^{-1}w^{-3/2}\right\}.$$
This means that $\Err(w) \lll 1$ holds as long as $w \ggg n^{-2/3}$, and the high-low estimate (\ref{eq:high_low_bound}) shows that 
\begin{equation*}
	B(w; P, L) \sim 1 \text{ for all $w \gtrsim n^{-2/3}$}
\end{equation*}
which is the sharp regime. This example agrees with our earlier claim that $(|P||L|)^{-1/3} = n^{-2/3}$ is the smallest scale the high-low method can reach. 

We can take the same example and perturb each line slightly so that it's as far as possible from the point set $P$. Then each line will have distance $n^{-2/3}$ from any point, and we will have  

\begin{equation*}
	B(w; P, L) \sim \begin{cases}
		1 & \quad w > n^{-2/3}, \\ 
		0 & \quad w < n^{-2/3}. 
	\end{cases}
\end{equation*}

Once again the high-low estimate gives a lower bound on incidences in the sharp range of scales. 

\bigskip

{\bf{Lower bounds and upper bounds for incidences}}.
Roth developed his inductive step to prove lower bounds for incidences. Guth, Solomon, and Wang \cite{GSW} introduced the high-low method to prove upper bounds for incidences. Theorem \ref{thm:high_low_bound} unifies these two perspectives: the same inequality can be used to prove upper bounds or lower bounds. Lots of famous problems in combinatorics and analysis---for example, the unit distance problem and the Kakeya problem---have to do with upper bounds for incidences, and there is a huge amount of research on this topic. On the other hand, incidence lower bounds have not been studied much. 
We hope to draw attention that these lower bounds are the key in Heilbronn's problem, and to encourage more research on lower bounds for incidences in general.

\subsection{Proof of the high-low bound}\label{subsec:high_low_proof}
In this section we prove the high-low estimate, Theorem \ref{thm:high_low_bound}.

Let $P = \{x_1, \ldots, x_n\}$ be a set of points in $[0,1]^2$, $L = \{\ell_1, \ldots, \ell_m\}$ a set of lines. 

Let
\begin{align*}
    f_w = \sum_{\ell \in L} \frac{1}{w} 1_{\T_{\ell}(w)},\quad g = \sum_{p \in P} \delta_p. 
\end{align*}
Then 
\begin{equation*}
    \langle f_w, g\rangle = w^{-1}\, \#\{(p, \ell) \in P \times L\, :\, d(p, \ell) \leq w/2\}. 
\end{equation*}
We are ready to give a smoothed definition of incidences.
Let $\chi: \R^2 \to \R_{\geq 0}$ be a fixed radially symmetric bump function with $\int \chi = 1$ and $\supp \chi \subset B_{1/50}$. Let $\chi_w(x) = w^{-2}\chi(w^{-1}x)$. We define the smoothed version of incidences by convolving $f_w$ with $\chi_w * \chi_{w/10}$, 
\begin{equation}\label{eq:defn_smoothed_incidences_inner_prod}
    I(w; P, L) := w\, \langle \chi_w * \chi_{w/10} * f_w, g\rangle. 
\end{equation}
Let 
\begin{equation*}
    \eta(t) = \int_{-1/2}^{1/2} \int_{-\infty}^{\infty}(\chi * \chi_{1/10})(t+r, s)\, ds\, dr. 
\end{equation*}
Notice that $\eta|_{[-1/2+1/10, 1/2-1/10]} = 1$ and $\supp \eta \subset [-1/2-1/10, 1/2+1/10]$. 
Because $\chi$ was chosen radially symmetric, we see that 
\begin{equation*}
    I(w; P, L) = \sum_{(p, \ell) \in P \times L} \eta(w^{-1}d(p, \ell))
\end{equation*}
as we claimed in (\ref{eq:defn_smoothed_incidences_formula}). In the analysis we will use (\ref{eq:defn_smoothed_incidences_inner_prod}) rather than (\ref{eq:defn_smoothed_incidences_formula}).
We have 
\begin{equation}
    B(w) = \frac{I(w)}{w\, |P|\, |L|} = \frac{1}{|P|\, |L|}\langle \chi_w * \chi_{w/10} * f_w, g\rangle,
\end{equation}
hence
\begin{align*}
    B(w) - B(w/10) &= \frac{1}{|P|\, |L|}\, \langle \chi_{w} * \chi_{w/10} *  f_w  - \chi_{w/10} * \chi_{w/100} * f_{w/10} , g\rangle \\ 
    &= \frac{1}{|P|\, |L|}\, \langle \chi_{w/10} * (\chi_{w} * f_w  - \chi_{w/100} * f_{w/10}) , g\rangle.
\end{align*}
We convolved $f_w$ with $\chi_w * \chi_{w/10}$ in (\ref{eq:defn_smoothed_incidences_inner_prod}) rather than just $\chi_w$ in order to factor out $\chi_{w/10}$ in this equation. Putting the convolution on the other side and using Cauchy-Schwarz,
\begin{equation}
    |B(w) - B(w/10)| \leq \frac{1}{|P|\, |L|}\, \| \chi_w * f_w - \chi_{w/100} * f_{w/10} \|_{L^2([-2,2]^2)}\, \| \chi_{w/10} * g \|_{L^2([-2,2]^2)}. 
\end{equation}
We estimate the term involving $g$ using the quantity $M_P(w\times w)$, 
\begin{align*}
    \| \chi_{w/10} * g \|_2^2 \leq \| g \|_1 \| \chi_{w/10} * g \|_{\infty} \lesssim w^{-2} |P|\,M_P(w\times w).
\end{align*}
We estimate the term involving $f$ using orthogonality and the quantity $M_L(w\times 1)$. 
Let
\begin{align*}
    \Phi_{\ell} &= \chi_w * w^{-1}1_{\T_{\ell}(w)} - \chi_{w/100} *  (w/10)^{-1}1_{\T_{\ell}(w/10)}\\ 
    \sum_{\ell \in L} \Phi_{\ell} &= \chi_w * f_w - \chi_{w/100} * f_{w/10}.
\end{align*}
We expand the $L^2$ norm as a sum of truncated inner products. Let $\psi: \R^2 \to [0,1]$ be a fixed smooth bump function with $\psi|_{[-2,2]^2} = 1$ and $\supp \psi \subset [-3,3]^2$. We have
\begin{equation*}
    \Bigl \| \sum_{\ell \in L} \Phi_{\ell} \Bigr \|_{L^2([-2,2]^2)}^2 \leq \sum_{\ell_1, \ell_2} \left|\int \Phi_{\ell_1}(x)\Phi_{\ell_2}(x)\, \psi(x) dx\right|
\end{equation*}
If $\ell_1$ and $\ell_2$ are not parallel, then $\langle \Phi_{\ell_1}, \Phi_{\ell_2}\rangle = 0$ where the inner product is taken in all of $\R^2$. 

Here is an approximate version of this orthogonality. Let $\ell_1, \ell_2$ be two lines. Say $\ell_1 \sim \ell_2$ if 
\begin{equation*}
    \T_{2w}(\ell_1)\cap \T_{2w}(\ell_2) \cap [-3,3]^2 \neq \emptyset.
\end{equation*}
If $\ell_1 \not\sim \ell_2$ then the truncated inner product vanishes. 
Let $\ell_1\sim \ell_2$ make an angle $0 < \theta \leq \pi/2$. 
Here is an easy estimate for the truncated inner product:
\begin{equation*}
    \left|\int \Phi_{\ell_1}(x)\Phi_{\ell_2}(x)\, \psi(x) dx\right| \leq \| \Phi_{\ell_1} \|_{L^2([-3,3]^2)} \| \Phi_{\ell_2} \|_{L^2([-3,3]^2)} \lesssim w^{-1}.
\end{equation*}
This is good enough when $\theta \leq w$, but we need an improvement using orthogonality when $\theta$ is large. Let $R = \T_{2w}(\ell_1) \cap \T_{2w}(\ell_2)$, and let $x_R = \ell_1 \cap \ell_2$ be the center of $R$. Notice that $R$ is a parallelogram which is contained in a rectangle with width $w$ and length $2w/\sin \theta$. 
Using the fact that $\Phi_{\ell_1}$ and $\Phi_{\ell_2}$ are orthogonal and re-centering to $R$, 
\begin{equation*}
    \int \Phi_{\ell_1}(x) \Phi_{\ell_2}(x)\, \psi(x) dx = \int (\Phi_{\ell_1}\Phi_{\ell_2})(x_R+y)\, (\psi(x_R+y) - \psi(x_R))\, dy. 
\end{equation*}
Because $\psi$ is a fixed smooth function, we have the linear approximation 
\begin{equation*}
    \psi(x_R+y) - \psi(x_R) = \nabla \psi(x_R) \cdot y + O(|y|^2).
\end{equation*}
The integral against the linear part vanishes due to the symmetry
\begin{equation*}
    \Phi_{\ell_1}(x_R+y) = \Phi_{\ell_1}(x_R-y)\quad \text{and}\quad \Phi_{\ell_2}(x_R+y) = \Phi_{\ell_2}(x_R-y).
\end{equation*}
The $O(|y|^2)$ term is bounded by $C\cdot \mathrm{Diameter}(R)^2$, so 
\begin{equation*}
    \left|\int \Phi_{\ell_1}(x) \Phi_{\ell_2}(x)\, \psi(x) dx\right| \lesssim \Area(R)\, \mathrm{Diameter}(R)^2\, w^{-2} \lesssim (w^2\theta^{-1})\, (w\theta^{-1})^2\, w^{-2} = w^2\theta^{-3}.
\end{equation*}
Next, we have the estimate 
\begin{equation*}
    \#\{\ell_1 \sim \ell_2\, :\, \theta(\ell_1, \ell_2) \in [\theta_0, \theta_0 + w/100]\} \lesssim M_L(w\times 1)\, \max(1,\theta_0/w)
\end{equation*}
because any line $\ell_2$ in the set on the left has to be contained in one of $\lesssim \max(1,\theta_0/w)$ many $w$ tubes. Combining these estimates, we see that for fixed $\ell_1$, 
\begin{align*}
    \sum_{\ell_2 \in L} \left|\int \Phi_{\ell_1}(x)\Phi_{\ell_2}(x)\, \psi(x) dx\right| &\lesssim \sum_{j=1}^{200w^{-1}} w^2 (jw)^{-3}\, \#\{\ell_1 \sim \ell_2\, :\, \theta(\ell_1, \ell_2) \in [(j-1)w/100, jw/100]\} \\ 
    &\lesssim w^{-1}M_L(w\times 1)\, \sum_{j=1}^{200w^{-1}}  j^{-2} \lesssim w^{-1}M_L(w\times 1).
\end{align*}
Putting everything together,
\begin{equation*}
    |B(w) - B(w/10)| \lesssim  \sqrt{\frac{M_P(w\times w)}{|P|}\frac{M_L(w\times 1)}{|L|}\, w^{-3}}.
\end{equation*}

\section{Initial estimates via projection theory}\label{sec:inital_estimate}

In this section, we describe our new projection theory approach to initial estimates of the form \eqref{eq:roth_initial_estimate}. This approach is quite different from Roth's original argument from \cite{Roth2} (and the subsequent papers that used it) and is the main novelty of our paper. For a fixed scale $u > n^{-1/2}$, consider the pairs $\tau = \{x,y\} \in {P \choose 2}$ such that $d(\tau) \leq u$. For each such $\tau$, let $\ell_{\tau}$ denote the line supporting $\tau$. Partition $[0,1]^2$ into a grid of $u\times u$ squares. Within each of these squares $Q$, consider the set of points $P\cap Q$, and define the set of lines
\begin{equation*}
	L_Q = \left\{\ell_{\tau}\, :\, \tau \in {P \cap Q \choose 2}\right\}.
\end{equation*}
Then we make the line set $L = \bigcup_Q L_Q$. In order to prove an initial estimate of the form (\ref{eq:roth_initial_estimate}) for $P$ and $L$ at the scale $w_i$, we must show that for most of the lines $\ell \in L$, 
\begin{equation}\label{eq:line_is_good}
	\#\{p \in \T_{w_i}(\ell)\} \gtrsim w_i|P|. 
\end{equation}

Let us work one square $Q$ at a time. The initial scale $w_i$ will be much larger than $u$, so the tubes $\T_{w_i}(\ell)$ will contain the square $Q$. Note also that the tube $\T_{w_i}(\ell)$ is essentially determined by $\theta(\ell)$ up to an angular resolution of $\sim w_i$. Say that a direction $\theta \in \left\{w_i/100,2w_i/100,\ldots,1\right\}$ is {\it{good}} if the $w_i \times 1$ tube in direction $\theta$ containing $Q$ has $\gtrsim w_i |P|$ points of $P$, otherwise say $\theta$ is bad. Using a double counting argument, we can ensure that for most of the squares $Q$ in the partition there are not-too-many bad directions. The main challenge is to show that the set of directions spanned by the lines in $L_Q$ is not concentrated in the small number of bad directions. 
\begin{figure}[h]
\centering
\includegraphics[scale=0.25]{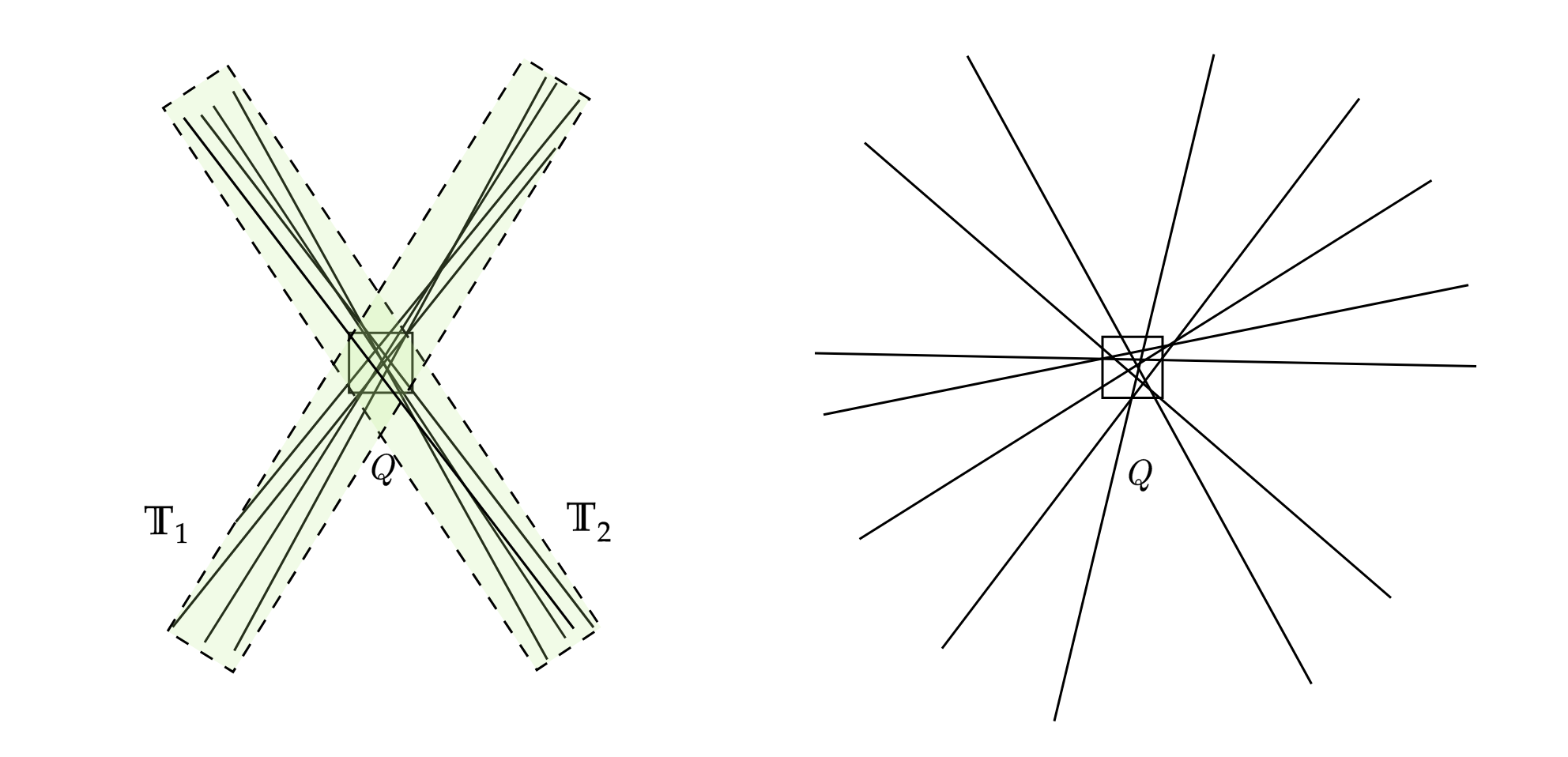}
\caption{A set of lines $L_Q$ spanned by points inside a box $Q$. On the left, the directions $\theta(\ell)$ are concentrated in a few tubes $\mathbb{T}_i$ and we cannot prove an initial estimate. On the right, the directions are spread out and many lines avoid `bad' tubes.}
\label{fig:tubes}
\end{figure}

A crucial observation is that if $P \cap Q$ satisfies an $s$-dimensional Frostman condition for some $s > 1$ then a discretized version of Theorem \ref{Mars} implies that the direction set is well spread out.
Before we state the Theorem, here is an important definition.

\bigskip

{\bf{Definition}}. Let $P \subset [0,1]^2$ be a set of points. We say $P$ is {\it{$s$-regular at all scales with constant $C$}} if for all $w\times w$ squares $Q$, the following inequality holds:
\begin{equation*}\label{eq:s_regularity_pts_all}
    |P \cap Q| \leq C\max\left\{1, |P|w^s\right\}. 
\end{equation*}

We also say $P$ is {\it{$s$-regular above scale $\delta$ with constant $C$}} if
\begin{equation}\label{eq:s_regularity_pts_above_delta}
    |P\cap Q| \leq C |P|w^s\quad \text{ for all $w > \delta$}. 
\end{equation}

A variant of these definitions goes back to the work of Katz and Tao \cite{KT01}.

\begin{theorem}\label{marstrand}
Let $s > 1$, $\sigma < 1$, and $C > 0$. There exists $K = K(C, s, \sigma)$ such that the following holds. Let $P \subset [0,1]^2$ be $s$-regular above scale $\delta$ with constant $C$. Then there is a set of lines 
\begin{equation*}
	L \subset \{\ell_{x,y}\, :\, x, y \in P\}
\end{equation*}
formed from pairs in $P$ such that $|L| \geq K^{-1} |P|^2$ and 
\begin{equation*}
	\#\{\ell \in L\, :\, \theta(\ell) \in I\} \leq K |L|\, |I|^{\sigma},\quad \text{for all intervals $I \subset S^1$ with $|I| \geq \delta$}. 
\end{equation*}
\end{theorem}

In order to use Theorem \ref{marstrand} we need to choose the set of squares more carefully: instead of starting with a partition of $[0,1]^2$ into $u \times u$ squares we will greedily construct a collection $\{Q\}_{Q \in \mathcal{Q}}$ of pairwise disjoint squares such that each $P \cap Q$ is $s$-dimensional for some fixed $s > 1$. 
Each set of lines $L_{Q}$ thus spans a well spread out set of directions, which ensures \eqref{eq:line_is_good} holds for most $\ell \in L$. It follows that $B(w; P, L) \gtrapprox 1$ for some initial scale $w$ and a set of lines $L$ spanned by pairs of points from $P$ at distance at most $u$. Here the distance $u$ is the most popular side length among the squares from $\mathcal Q$. 
Because of the way we construct $\mathcal Q$, we do not get to choose a specific value of $u$ but have to work with an arbitrary $u$ in a certain range.

By combining the lower bound $B(w; P, L) \gtrapprox 1$ with the high-low method we can recover the classical bound $\Delta(n) \lessapprox n^{-8/7}$ of Koml\'os--Pintz--Szemer\'edi. We present full details of this argument in Section \ref{KPSsection} as a warm-up before the proof of the improved bound. 

\bigskip

To improve on the $8/7$ exponent we use the following discretized version of Orponen, Shmerkin, and Wang's Theorem \ref{thm:OSW_dir_set}.

\begin{theorem}\label{thm:finite_set_dir_set}
For all $0 < s < 1 $, $\sigma < s$, and $C, w > 0$ there exists $\tau = \tau(\sigma, s) > 0$ and $K = K(C, w, \varepsilon, s, \sigma) < 0$ such that the following holds. 
Let $P \subset [0,1]^2$ be $s$-regular above scale $\delta$ with constant $C$. Suppose 
\begin{equation}\label{eq:finite_set_tube_cond}
	|P\cap T| \leq \tau |P|
\end{equation}
for all $w$-tubes $T$. There is a set of lines
\begin{equation*}
	L \subset \{\ell_{x,y}\, :\, x, y \in P\}
\end{equation*}
formed from pairs in $P$ such that $|L| \geq K^{-1}|P|^2$ and
\begin{equation*}
	\#\{\ell \in L\, :\, \theta(\ell) \in I\} \leq K|L|\, |I|^{\sigma},\quad \text{$I \subset S^1$ an interval, $|I| \geq \delta$}.
\end{equation*}
\end{theorem}
Both Theorem \ref{marstrand} and Theorem \ref{thm:finite_set_dir_set} follow from {\cite[Corollary 2.18]{OSW}} by relatively standard arguments. We sketch this deduction in Appendix \ref{app:proof_of_dir_set_estimates}. 

By using Theorem \ref{thm:finite_set_dir_set} in place of Theorem \ref{marstrand}, we get an initial estimate for the line set $L_Q$ even if $P\cap Q$ is less than one dimensional. In the end we use this edge to improve the overall bounds, but there are significant technicalities. 

For one, the condition (\ref{eq:finite_set_tube_cond}) adds a complication: it is possible that the intersection $P\cap Q$ is essentially contained in a tube $T$ of width $w < u$. In this case the direction of lines in $L_Q$ is concentrated near $T$, so we get a weaker initial estimate. On the one hand, if $w$ is fairly large, the initial estimate is still good enough to run our argument. On the other hand, if $w$ is very small then we can find a triangle of a small area inside the set $P\cap Q$ alone. 
Unfortunately, the `fairly large' and `very small' regimes for $w$ do not always overlap. It turns out that if the initial distance scale $u$ is below a certain threshold $u_1$ then the conditions on $w$ do overlap and we get an improved bound. 
To deal with this technicality we split into two cases, the \textit{square case} and the \textit{rectangle case}. We greedily pick a collection of squares $\mathcal Q$ with sides $\ge u_1$ and a collection of rectangles $\mathcal R$ with larger side $\le u_1$ such that: for $Q \in \mathcal Q$ the set $P \cap Q$ is $s_1$-dimensional for some fixed $s_1 > 1$ and for $R \in \mathcal R$ the set $P \cap R$ is $s_2$-dimensional for some fixed $0<s_2 < 1$.

\section{Lower bounds for incidences}\label{sec:lower_bounds_for_incidences}

\subsection{Statement of results}

Here is a basic lower bound on the number of incidences.

\begin{lemma}[Basic initial estimate]\label{lem:basic_initial_estimate_weighted}
Let $P \subset [0,1]^2$ be a set of points. For each $p \in P$, let $L_p$ be a set lines through $p$ such that $|L_p|$ is the same for all $p$, and 
\begin{equation}
    \#\{\ell \in L_p\, :\, \theta(\ell) \in I\} \leq \mu |L_p|\quad \text{for any interval $I \subset S^1$ with $|I| = \eta$.}\label{eq:basic_initial_not_concentrated}
\end{equation}
Let $L = \bigcup_p L_p$. Then 
\begin{equation*}
	\frac{|P \cap \T_{\eta}(\ell)|}{\eta |P|} \gtrsim \mu^{-1}\eta
\end{equation*}
for $\geq \frac{1}{2}|L|$ lines $\ell \in L$. 
\end{lemma}

For example, suppose that for each $p$, $L_p$ is a set of $\lfloor \eta^{-1}/1000\rfloor$ many $\eta$-separated lines. Then we can take $\mu = 1000\eta$, and we obtain $B(\eta; P, L) \gtrsim 1$. 

The following proposition gives an incidence lower bound for points and lines which are sufficiently spread out at scales $w_f$, $\eta$ respectively, as long as $w_f$ is not too small.

\begin{proposition}\label{prop:incidence_lower_bound_sep_weighted}
Let $0 < w_f < \eta < 1$ be parameters. 
Let $P \subset [0,1]^2$ be a set of points and for each $p \in P$, let $L_p$ be a set of lines through $p$ such that $|L_p|$ is the same for all $p$. Suppose that for any $w_f \times w_f$ square $Q$, 
\begin{equation}
    |P\cap Q| \leq \mu_P\, |P|.\label{eq:lower_bd_prop_upper_bd_pts_in_square}
\end{equation}
Suppose that for any $p \in P$, 
\begin{equation}
    \#\{\ell \in L_p\, :\, \theta(\ell) \in I\} \leq \mu_L |L_p|\quad \text{for any interval $I \subset S^1$ with $|I| = \eta$.} \label{eq:line_concentration_prop_sep}
\end{equation}
Let $L = \bigcup_p L_p$. There is a constant $c_0 > 0$ such that if 
\begin{equation}
    \mu_P\, \mu_L^2 \leq c_0\,e^{-|\log  w_f|^{0.9}}\, w_f^2 \eta \label{eq:mu_Pmu_Lhypothesis}
\end{equation}
holds, then
\begin{equation*}
    B(w_f; P, L) \gtrsim e^{-|\log w_f|^{0.6}}\, \mu^{-1}_L \eta. 
\end{equation*}
\end{proposition}

For example, suppose $P$ is a set of $\delta$-separated points and $L_p$ is a set of $\sim \eta^{-1}$ many $\eta$ separated lines through each $p$, with $\eta \ggg \delta$. Then we can take $\mu_P = \max((w_f/\delta)^2, |P|^{-1})$ and $\mu_L = \eta$, and the condition reads 
\begin{align*}
    w_f \ggg \sqrt{\eta / |P|}.  
\end{align*}
The Szemer\'edi-Trotter example from \S\ref{sec:inductive_step_high_low} is of this form, in that case we have $\eta = |P|^{-1/3}$ and $w_f = |P|^{-2/3}$ in agreement with our computation there.

The proof of Proposition \ref{prop:incidence_lower_bound_sep_weighted} involves an initial estimate using Lemma \ref{lem:basic_initial_estimate_weighted} and an inductive step using the high-low method (\ref{eq:high_low_bound}). 

\begin{remark}
We emphasize an important difference between using the high-low method to prove upper bounds vs. lower bounds. An upper bound looks like: $B(w_f) \lesssim B(w_i) + \mathrm{\Err}$. The initial estimate and the inductive step independently contribute to the final bound. A lower bound looks like: $B(w_f) \gtrsim B(w_i) - \mathrm{\Err}$. We need $\mathrm{\Err} \lll B(w_i)$ to get any lower bound at all. For this reason proving lower bounds is more delicate, and the basic initial estimate is absolutely crucial. 
\end{remark}

\begin{remark} \label{shift}
The conclusion of Proposition \ref{prop:incidence_lower_bound_sep_weighted} remains true when the lines from $L_p$ are at distance $< w_f/5$ from $p$ rather than all concurrent at $p$. Indeed, if that's the case, consider the set of lines $L_p'$ obtained by translating each line in $L_p$ to pass through $p$. Then,
\begin{equation*}
    \T_{w_f/3}(\ell') \subset \T_{w_f}(\ell)
\end{equation*}
holds for each $\ell \in L_p$, where $\ell'$ denotes its translation through $p$.
Thus
\begin{equation*}
    B(w_f; P, L) \geq \frac{1}{3}B(w_f/3; P, L') \gtrsim e^{-|\log w_f|^{0.6}}\, \mu^{-1}_L \eta. 
\end{equation*}
\end{remark}

\subsection{Proof of Lemma \ref{lem:basic_initial_estimate_weighted}} 

An important intuition is that there are only $\sim w^{-2}$ different $w\times 1$ tubes. Before making this more precise, recall that for any two tubes $T$ and $T'$ the notation $T' \subset T$ stands for $T' \cap [0,1]^2 \subset T$ (and the same goes for line containment).

Let $w > 0$ be a parameter. There is a family of $w/2$-tubes $\mc T^w = \{T_1, \ldots, T_m\}$ such that 
\begin{enumerate}
    \item $|\mc T^w| \sim w^{-2}$,
    \item Let $T \subset \R^2$ be a $w$-tube, and let $\ell$ be the central line in $T$. There is some $T' \in \mc T^w$ such that 
    \begin{align*}
        \ell \subset T'  \subset T 
    \end{align*}
\end{enumerate}

One can define such a $\mathcal T^w$ as follows. Consider the family of tubes 
$$
\mathcal T_0 = \left\{\Bigl[\frac{j w}{100}, \frac{jw}{100} + \frac{w}{2}\Bigr] \times \R, ~~ \text{for }j =-200w^{-1}, \ldots, 200 w^{-1}\right\}.
$$ 
Now take $\mathcal T^w$ to be a union of rotations of $\mathcal T_0$ around the point $(1/2,1/2)\in [0, 1]^2$ by angles $\theta = jw/100$, for $j=0, \ldots, 200 w^{-1}$. The set $\mathcal{T}^{w}$ satisfies the above conditions. 

We are ready to prove Lemma \ref{lem:basic_initial_estimate_weighted}. Recall that $P \subset [0,1]^2$ is a set of $n$ points and that for each $p \in P$ we have a pencil of lines $L_p$ such that $|L_p|$ is the same for all $p$, and 
\begin{equation}
    \#\{\ell \in L_p\, :\, \theta(\ell) \in I\} \leq \mu |L_p|\quad \text{for any interval $I \subset S^1$ with $|I| = \eta$.}\label{eq:basic_initial_not_concentrated}
\end{equation}
Furthermore, $L = \bigcup_p L_p$. Let $\lambda > 0$ be a parameter to be chosen later. Say a tube $T$ is \textit{bad} if $\frac{|P \cap T|}{\eta |P|} \leq \lambda$. Say $\ell \in L$ is bad if $T_{\eta}(\ell)$ is bad. If a tube or line is not bad we say it is \textit{good}. 

Let $\mc T^{\eta}$ be a collection of $(\eta/2)$-tubes as above, and let $\mc T^{\eta}_{bad}$ be the set of bad tubes in $T^{\eta}$.  Let $\ell \in L$. There is some $T \in \mc T^{\eta}$ such that 
\begin{equation*} \label{cover}
    \ell \subset T \subset \T_{\eta}(\ell),
\end{equation*}
and if $\ell$ is bad then $T$ is bad. Using this idea, we estimate the number of bad lines $\ell \in L$ as follows:
\begin{align*}
    \#\{\ell \in L\, :\, \text{$\ell$ is bad}\} &\leq \sum_{T \in \mc T^{\eta}_{bad}} \sum_{p \in T} \#\{\ell \in L_p\, :\, T \subset \T_{\eta}(\ell) \}.
\end{align*}
Let $\ell \in L$. If $\T_{\eta}(\ell)\supset T$ then $|\theta(\ell) - \theta(T)| \leq 50\eta$. From the hypothesis (\ref{eq:basic_initial_not_concentrated}), we know that
\begin{equation*}
    \#\{\ell \in L_p\, :\, \T_{\eta}(\ell)\supset T\} \leq  100 \mu  |L_p|. 
\end{equation*}
Using this estimate and the fact that bad tubes have $\leq \lambda \eta |P|$ points, we have
\begin{equation*}
    \#\{\ell \in L\, :\, \text{$\ell$ is bad}\} \lesssim |\mc T^{\eta}|\, (\lambda \eta |P|) \, (\mu |L_p|) \lesssim \lambda \mu \eta^{-1}\, |L|. 
\end{equation*}
Thus if we choose $\lambda \leq c\, \mu^{-1} \eta$ for small enough $c$ then at least $\frac{1}{2}|L|$ lines in $L$ are good, as desired. Note that in this step we used the fact that $|L_p|$ is the same for every $p$ to write $|P| |L_p| = |L|$. 

\subsection{Proof of Proposition \ref{prop:incidence_lower_bound_sep_weighted}}

Let $P, L$ be as in Proposition \ref{prop:incidence_lower_bound_sep_weighted}. We will use the basic initial estimate Lemma \ref{lem:basic_initial_estimate_weighted} to obtain an estimate at scale $\eta$, and then we will use the high-low estimate (\ref{eq:high_low_bound}) to get a final estimate at scale $w_f$. First we establish some point and line regularity. 

We can cover any $w\times w$ square with $\lceil w / w_f\rceil^2$ many $w_f\times w_f$ squares, so using the hypothesis (\ref{eq:lower_bd_prop_upper_bd_pts_in_square}), 
\begin{equation}\label{eq:inc_lower_bd_pt_reg}
      M_P(w\times w) \leq 2\left(\frac{w}{w_f}\right)^2\, \mu_P\, |P|. 
\end{equation}

Let $T$ be a $w\times 1$ tube, $w \leq \eta$. Recall that $T\cap L = \{\ell \in L\, :\, \ell \subset T\}$. By the hypothesis (\ref{eq:line_concentration_prop_sep}), 
\begin{equation}\label{eq:line_concentration_sep_tube}
      |T\cap L| \leq \sum_{p \in T} |T\cap L_p| \lesssim |T\cap P|\, \mu_L\, |L_p|.  
\end{equation}
Let us make a few comments about this estimate. Our goal is to prove that there are lots of incidences at some scale $w < \eta$. On the one hand, if $T\cap P$ is large for many tubes $T = \T_w(\ell)$ then there are a lot of incidences and we are in good shape. On the other hand, larger $T\cap P$ leads to a larger error in the high-low estimate. To deal with this trade-off we dyadically pigeonhole based on the number of points in a tube.

The following Lemma is the single scale ingredient to our inductive step. 

\begin{lemma}\label{lem:single_scale_estimate_prop_separated}
Suppose $w \geq w_f$. If 
\begin{equation}\label{eq:single_scale_lem_hyp}
     B(w; P, L) \geq C K^3|\log w_f|\, \mu_P \mu_L w_f^{-2}
\end{equation}
then
\begin{equation*}\label{eq:single_scale_estimate}
      B(w/K; P, L) \geq C^{-1} |\log w_f|^{-1} B(w; P, L). 
\end{equation*}
Here, $C$ is a universal constant. 
\end{lemma}

Using this Lemma along with the basic initial estimate, we inductively prove the following. 

\begin{lemma}\label{lem:inductive_step_estimate_prop_separated}
Let $K = e^{|\log w_f|^{0.5}}$. For any $j \geq 0$ so that $\eta / K^j \geq w_f/K$, we have
\begin{equation}\label{eq:several_scale_estimate}
      B(\eta / K^j ; P, L) \geq (C |\log w_f|)^{-j-1}\, \mu_L^{-1} \eta. 
\end{equation}
\end{lemma}

Notice that (\ref{eq:several_scale_estimate}) shows 
$$B(w_f; P, L) \geq (C |\log w_f|)^{-2\frac{\log (\eta / w_f)}{\log K}}\, \mu_L^{-1} \eta \geq (C |\log w_f|)^{-|\log w_f|^{0.5}}\, \mu_L^{-1}\eta \gtrsim e^{-|\log w_f|^{0.6}}\, \mu_L^{-1}\eta,$$
as needed.
Now we prove the lemmas. 

\begin{proof}[Proof of Lemma \ref{lem:single_scale_estimate_prop_separated}]

For a scale $w > w_f$ and $j \in \Z$, let 
\begin{equation*}
      L^j = \Bigl\{\ell \in L\, :\, 2^{j-1} \leq \frac{|\T_w(\ell)\cap P|}{w\, |P|} \leq 2^j\Bigr\}.
\end{equation*}
Let $L^* = L^j$ with $j$ chosen to maximize $I(w; P, L^{j})$. Notice that for any $j \in \Z$, 
\begin{equation*}
    I(w; P, L^j) \leq 2^{j} w |P| |L|.
\end{equation*}
Let $j_0 < 0$ be the largest integer such that $2^{j_0} \leq \frac{1}{1000}B(w; P, L)$. Then using the hypothesis (\ref{eq:single_scale_lem_hyp}), 
\begin{equation*}
    j_0 \gtrsim \log B(w; P, L) \gtrsim \log \eta + \log \mu_P \gtrsim -|\log w_f|. 
\end{equation*} 
Also, because $|\T_w(\ell)\cap P| \leq |P|$, $L^j$ is empty for $j \gtrsim |\log w_f|$. Thus 
\begin{equation*}
    \sum_{-C |\log w_f| \leq j \leq C |\log w_f|} I(w; P, L^j) \geq \frac{1}{2} I(w; P, L),
\end{equation*}
so by the pigeonhole principle,
\begin{align}
    I(w; P, L^*) &\gtrsim |\log w_f|^{-1}\, I(w; P, L),\nonumber\\
    B(w; P, L^*) &\gtrsim  |\log w_f|^{-1} B(w; P, L) \frac{|L|}{|L^*|}.\label{eq:B_est_L*}
\end{align}
By equation (\ref{eq:line_concentration_sep_tube}),
\begin{align}
      M_{L^{j}}(w \times 1) &\lesssim (2^j w|P|)\, (\mu_L |L_p|) = 2^j w \mu_L\, |L|, \nonumber\\ 
      M_{L^*}(w\times 1)&\lesssim w\mu_L\, |L|\, B(w; P, L^*).\label{eq:line_conc_L*}
\end{align}
Now we estimate the error term (\ref{eq:defn_err_equation}) of the high-low method,
\begin{align*}
      \left(\frac{\Err(w; P, L^*)}{B(w; P, L^*)}\right)^2 &= \frac{M_P(w\times w)}{|P|} \frac{M_{L^*}(w\times 1)}{|L|} w^{-3} B(w; P, L^*)^{-2} \\ 
      &\lesssim w_f^{-2}w^2\, \mu_P\, \mu_L w \frac{|L|}{|L^*|}\, \, w^{-3}\, B(w; P, L^*)^{-1} && \text{by (\ref{eq:inc_lower_bd_pt_reg}) and (\ref{eq:line_conc_L*})}\\ 
      &\lesssim w_f^{-2}\, \mu_P\, \mu_L\, \left(\frac{|L^*|}{|L|}\, B(w; P, L^*)\right)^{-1} \\ 
      &\lesssim  |\log w_f|\, w_f^{-2}\, \mu_P\, \mu_L\, B(w; P, L)^{-1} && \text{by (\ref{eq:B_est_L*})}\\ 
      &\leq C^{-1}K^{-3} && \text{by (\ref{eq:single_scale_lem_hyp})}
\end{align*}
By the high-low bound (\ref{eq:high_low_bound_K}), we obtain
\begin{equation*}
      B(w/K; P, L^*) \geq \frac{1}{2} B(w; P, L^*).
\end{equation*}
We can obtain a bound with $L^*$ replaced with $L$ as follows:
$$B(w/K; P, L) \geq B(w/K; P, L^*) \frac{|L^*|}{|L|} \gtrsim B(w; P, L^*) \frac{|L^*|}{|L|} \gtrsim |\log w_f|^{-1} B(w; P, L).$$
\end{proof}

\begin{proof}[Proof of Lemma \ref{lem:inductive_step_estimate_prop_separated}]
First, by Lemma \ref{lem:basic_initial_estimate_weighted}, 
\begin{equation*}
      B(\eta; P, L) \gtrsim \mu_L^{-1} \eta. 
\end{equation*}
This proves (\ref{eq:several_scale_estimate}) for $j = 0$. 
By the hypothesis (\ref{eq:mu_Pmu_Lhypothesis}), 
\begin{equation*}
    \mu_L^{-1}\eta \geq \mu_P \mu_L w_f^{-2}\, c_0^{-1} e^{|\log w_f|^{0.5}}.
\end{equation*}
If $\eta / K^j \geq w_f / K$ then $j \leq \lceil \frac{|\log w_f|}{\log K}\rceil \lesssim |\log w_f|^{0.5}$. Then (\ref{eq:several_scale_estimate}) implies 
\begin{align*}
    B(\eta / K^j; P, L) &\geq (C|\log w_f|)^{-j-1}\, \mu_P \mu_L w_f^{-2}\, c_0^{-1} e^{|\log w_f|^{0.9}} \\ 
    &\geq c_0^{-1}\, e^{|\log w_f|^{0.9} - |\log w_f|^{0.5}\log \log w_f^{-1} - C}\, \mu_P\, \mu_L\, w_f^{-2}\\
    &\geq CK^3\, |\log w_f|\, \mu_P\mu_L w_f^{-2}
\end{align*}
as long as the small constant $c_0$ in the statement of Proposition \ref{prop:incidence_lower_bound_sep_weighted} is small enough. Apply Lemma \ref{lem:single_scale_estimate_prop_separated} to finish the proof.
\end{proof}

\section{Improved bounds for the Heilbronn triangle problem}\label{sec:improved_bd_proof}

\subsection{From incidences to small triangles}

The following lemma will be used to deduce the improved bound for Heilbronn's triangle problem. Given a lower bound on the number of incidences between the set of points $P$ and a set of lines spanned by pairs of points of $P$ at certain distance, this result asserts that $P$ contains triangles of small area. 

We need a definition that will be used in the proof. Say two $w$-tubes $T, T'$ are \textit{essentially distinct} if $T' \not\subset 2T$ where $2T$ is the central dilate (recall that we intersect the tubes with $[0,1]^2$ before evaluating the subset relation). Let $w_2 < w_1$ and let $T$ be a fixed $w_1$ tube. If $\mc T = \{T_1, \ldots, T_m\}$ is a collection of essentially distinct $w_2$ tubes with $T_j \subset T$, then $|\mc T| \lesssim (w_1/w_2)^2$. 

Recall that for $w > 0$ we use the notation $B(w) = \frac{I(w; P, L)}{w |P| |L|}$ for the normalized number of incidences between the set of points $P$ and the set of lines $L$ at scale $w$. 

\begin{lemma}\label{lem:low_scales}
Let $P \subset [0, 1]^2$ be a set of points and let $\kappa > 0$ be a parameter. Suppose that for some $u$ and $w$ there exists a collection of lines $L$ spanned by pairs of points of $P$ at distance at most $u$ such that $B(w; P, L) \ge \kappa$, $I(w; P, L) > C |L|$, and $|L| > C |P|$.
Then $P$ contains a triangle of area $\Delta$ where
\begin{equation}\label{deltaeq}
\Delta/|\log \Delta| \lesssim \max\{u, u^{2/3}w^{1/3}\} |P|^{-1/3} |L|^{-1/3}\kappa^{-2/3} +  u |P|^{-1} \kappa^{-1}.    
\end{equation}
\end{lemma}

\begin{proof}
For the purposes of this section, let $\Delta$ denote the smallest area determined by three points from $P$. Suppose that $\Delta$ does not satisfy (\ref{deltaeq}) for a sufficiently large implied constant $C$.

First, recall that since the lines in $L$ are spanned by pairs of points of $P$ at distance at most $u$, for any line $\ell \in L$ the strip of width $4\Delta / u$ around $\ell$ contains exactly $2$ points of $P$. Otherwise, taking the pair of points $x, y \in P$ generating $\ell$ and a third point in the strip $\T_\ell(4\Delta/u)$ would create a triangle of area at most $\Delta$.
In particular, if we denote $w_f = \Delta / u$, then we have
\begin{equation}\label{eq:B_wf_upb_lem_low_scales}
B(w_f) = \frac{I(w_f; P, L)}{w_f |P| |L|} \le \frac{2 |L|}{w_f |P| |L|} \lesssim \frac{u}{\Delta |P|} < \kappa/C,
\end{equation}
where we used the assumption $\Delta > C u |P|^{-1}\kappa^{-1}$. 

Next, recall that $I(w_f; P, L)$ denotes the smoothed out number of incidences at scale $w_f$ (see (\ref{eq:defn_smoothed_incidences_formula})) so we take $w_f$ to be a constant factor smaller than $4\Delta/u$ to account for that.

We use the high-low method to lower bound $B(w_f)$ and get a contradiction. Recall that the error term \eqref{eq:defn_err_equation} in the high-low method depends on point and line regularity. 
Note that $P$ contains no three points with pairwise distances less than $\Delta^{1/2}$, since they would otherwise create a triangle of area at most $\Delta$. By a standard packing argument, it follows that $M_P(v \times v) \lesssim \max\{1, v^2 \Delta^{-1} \}$. Furthermore, note that for any distinct lines $\ell, \ell' \in L$ the tubes $\T_\ell(w_f/10)$ and $\T_{\ell'}(w_f/10)$ are essentially distinct. Otherwise the strip $\T_\ell(w_f)$ would have to contain the points of $P$ spanning $\ell'$ --and since $\ell \neq \ell'$, this means that $\T_\ell(w_f)$ would have to contain at least $3$ points of $P$ (which then generates a triangle of area less than $\Delta$, a contradiction).\footnote{This observation goes back to Schmidt, he used it to prove the estimate \eqref{bounded_deg} from Section \ref{sec:Incidences}.}. Because a $v$-tube contains $\lesssim (v/w_f)^2$ essentially distinct $w_f/10$ tubes, we get $M_L(v\times 1) \lesssim (v / w_f)^2$. 

We use our bounds on $M_P$, $M_L$ to estimate the high-low error:
\begin{equation}\label{erroreq}
\Err(v) \lesssim \sqrt{\frac{\max\{1, v^2 \Delta^{-1}\}}{|P|} \frac{(v/w_f)^2}{|L|} v^{-3}}.    
\end{equation}
Note that the right hand side of (\ref{erroreq}) is a convex function on the logarithmic scale and so it is maximized at one of the end points $v = w_f$ or $v = w$.

By the assumption that $|L| > C |P|$, there exist lines $\ell, \ell' \in L$ which intersect in a point of $P$. The points spanning these lines create a triangle of area at most $u^2$. So we have $u^2 \ge \Delta$ and since $w_f u = \Delta$ we get $w_f^2 \le \Delta$. We conclude
$$
\Err(w_f) \lesssim \sqrt{\frac{1}{|P| |L|} u^3 \Delta^{-3} } < \frac{\kappa}{C|\log \Delta|},
$$
where we used the assumption that $\Delta$ does not satisfy (\ref{deltaeq}). 

By the assumption that $I(w; P, L) >C |L|$, an average strip $\T_\ell(w)$ contains at least $3$ points of $P$. Thus, $P$ contains triangles of area at most $u w$ and we conclude that $\Delta \le u w$. Using this, we can bound 
$$
\max\{1, w^2 \Delta^{-1}\} \le \max\{uw \Delta^{-1}, w^2 \Delta^{-1}\} \le \max\{u, w\} w \Delta^{-1}
$$
and get 
$$
\Err(w) \lesssim  \sqrt{\frac{\max\{u, w\} w \Delta^{-1}}{|P|} \frac{(w/w_f)^2}{|L|} w^{-3}} = \sqrt{\frac{1}{|P||L|} \max\{u, w\} u^2 \Delta^{-3}} <  \frac{\kappa}{C|\log \Delta|},
$$
where in the last inequality we used the assumption on $\Delta$. By convexity of the right hand side of (\ref{erroreq}), we conclude that $\Err(v) <  \frac{\kappa}{C |\log \Delta|}$ for all $v \in [w_f, w]$. So using (\ref{multscale}), we conclude that $B(w_f) \gtrsim B(w) - \kappa/10 = 9\kappa/10$ which contradicts the lower bound (\ref{eq:B_wf_upb_lem_low_scales}).
Thus, our initial assumption on $\Delta$ is false and $P$ must contain a triangle of area less than $\Delta$ given by (\ref{deltaeq}).    
\end{proof}

\subsection{$s$-regular subsets relative to a rectangle}\label{Jsection}

Recall that a finite set of points $P \subset [0, 1]^2$ is $s$-regular above scale $\delta$ if for any square $Q \subset [0, 1]^2$ of width $w \ge \delta$ we have $|P \cap Q| \le C w^s |P|$. 
It is convenient to generalize this definition to sets contained in an arbitrary rectangle $R \subset \R^2$ instead of the square $[0, 1]^2$. 
Namely, for a rectangle $R$ let $P \subset R$ be a finite set and let $\psi:[0, 1]^2 \rightarrow R$ be an affine isomorphism between the unit square and the rectangle $R$. We say that $P \subset R$ is {\it{$s$-regular above scale $\delta$ relative to $R$}} if $\psi^{-1}(P) \subset [0, 1]^2$ is $s$-regular above scale $\delta$ in the usual sense. 

Let $R$ be a rectangle and $\psi:[0, 1]^2\rightarrow R$ an affine isomorphism between the unit square and $R$. Let $T \subset \R^2$ be a strip. We say that $T$ has {\it{width $w$ relative to $R$}} if the strip $\psi^{-1}(T)$ has (the usual) width equal to $w$.

Using these notions we will be able to apply the discretized projection theorems, Theorem \ref{marstrand} and Theorem \ref{thm:finite_set_dir_set}. For example, given a finite set $P \subset [0, 1]^2$ and a number $0< s < 2$, we can construct an $s$-regular subset in $P$ as follows. Among all squares $Q \subset [0, 1]^2$ with side $w \ge w_0$ for some fixed $w_0$, choose one maximizing the expression 
\begin{equation}\label{maxq}
|P \cap Q| w^{-s}.    
\end{equation}
Let $P_Q = P \cap Q$ and observe that by definition, the set $P_Q \subset Q$ is $s$-regular relative to $Q$ above scale $w_0/w$. Indeed, if $Q' \subset Q$ is a square with side $w' \in [w_0, w]$ then by the maximality of $Q$ we have
$$
|P\cap Q'| (w')^{-s} \le |P \cap Q| w^{-s},
$$
i.e. $|P_Q \cap Q'| \le (w'/w)^s |P_Q|$. In the coordinate system of $Q$ the square $Q'$ has side $w'/w$ so this condition means precisely that $P_Q \subset Q$ is $s$-regular above scale $w_0/w$.

In the actual proof of Theorem \ref{thm:main_theorem} will use a slightly more complicated way to construct $s$-regular sets. Namely, by restricting more carefully the set of squares over which we maximize (\ref{maxq}), we can get more information on $w$ and $|P_Q|$. 

Furthermore, in order to apply Theorem \ref{thm:finite_set_dir_set} to an $s$-regular set $P$ we also need to check that any strip of width $\tau = \tau(\sigma, s)$ does not contain too many points of $P$. However, this property does not follow from the construction above and we need the following modification. Given a finite set of points $P$ and $0< s< 1$ we consider all rectangles $R$ with sides $a \le b$ such that $a b \ge A$ and choose a rectangle $R$ maximizing the expression
$$
|P \cap R| a^{-\varepsilon} b^{-s+\varepsilon}. 
$$
Here $A$ and $\varepsilon$ are parameters of the construction. It can be easily checked that the resulting set $P_R = P\cap R$ is $s$-regular above scale $\sqrt{A/ab}$ relative to $R$ and satisfies the strip property needed for Theorem \ref{thm:finite_set_dir_set}: any strip $T$ of width $w$ relative to $R$ satisfies $|T \cap P_R| \lesssim w^\varepsilon |P_R|$. 

\subsection{Koml\'os--Pintz--Szemer\'edi bound in new language} \label{KPSsection}

In this section we show how our techniques recover the $8/7$ exponent of Koml\'os--Pintz--Szemer\'edi \cite{KPS81}.

Fix $\varepsilon'' >0$ and for large enough $n$, let $P \subset [0, 1]^2$ be an $n$-element set of points containing no triangles of area less than $\Delta$. Our aim is to show that $\Delta \lesssim n^{-8/7+\varepsilon''}$. There are four steps.
\begin{enumerate}
    \item Partition $P$ into a collection of small Frostman regular subsets. This is achieved with Lemma \ref{lem:single_step_prep_87} and Lemma \ref{lem:several_step_prep_87}. 
    \item Apply Theorem \ref{marstrand} to make a line set $L$ from pairs in these subsets that is well spread out at an initial scale $\eta$. By the basic initial estimate (Lemma \ref{lem:basic_initial_estimate_weighted}) there are lots of incidences at scale $\eta$. 
    \item Apply Proposition \ref{prop:incidence_lower_bound_sep_weighted} to see that there are lots of incidences at a scale $w \ll \eta$. 
    \item Apply Lemma \ref{lem:low_scales} to find small triangles. 
\end{enumerate}

The initial estimate is 1. and 2., and the inductive step is 3. and 4. Before starting the proof, let
\begin{itemize}
    \item $\varepsilon'' \ggg \varepsilon' \ggg \varepsilon > 0$ be a rapidly decreasing sequence of small constants, say, $\varepsilon = \varepsilon''/100$ and $\varepsilon' = \varepsilon''/10$,
    \item $u_0 = \Delta^{-1/2} n^{-1+\varepsilon'}$ be the distance scale for the initial estimate, 
    \item $s = 1+\varepsilon$ be a dimension parameter (for upcoming applications of Theorem \ref{marstrand}).
\end{itemize}

\begin{lemma}[Single step preparation]\label{lem:single_step_prep_87}
Let $P$, $u_0$, and $s$ be as above. For some $u \in [\Delta^{1/2}, u_0]$ there exists a $u \times u$ square $Q \subset [0,1]^2$ such that 
\begin{itemize}
    \item $P_{Q}:= P \cap Q$ is $s$-regular relative to $Q$ at all scales with constant $2$,
    \item We have $v^{-s}|P\cap Q'| \lesssim u^{-s} |P_Q| $ for any square $Q'$ with side $v \in [u, C u_0]$. Here we can take $C=100$. 
    In particular, $u^{-s} |P_Q| \gtrsim u_0^{2-s} |P|$. 
\end{itemize}
\end{lemma}

\begin{lemma}[Several step preparation]\label{lem:several_step_prep_87}
For $P$, $u_0$, and $s$ as above, there exists a set of pairwise disjoint squares $\mc Q = \{Q_j\}$, sets $P_{Q_j} \subset P\cap Q_j$, an integer $n_1$ and $u \in [\Delta^{1/2}, u_0]$ such that the following are satisfied. Denote $P' = \bigcup_j P_{Q_j}$.
\begin{itemize}
    \item $|P_{Q_j}| \in [n_1, 2n_1]$, $|Q_j| \in [u, 2u]$, and $|P'| \gtrsim \frac{n}{(\log n)^2}$,
    \item $P_{Q_j} \subset Q_j$ is $s$-regular relative to $Q_j$ at all scales with constant $2$,
    \item $v^{-s}|P'\cap Q'| \lesssim u^{-s} n_1$ for any square $Q'$ with side $v \in [u, Cu_0]$. Here we can take $C=100$. Furthermore, $u^{-s} n_1 \gtrsim u_0^{2-s} n$.
\end{itemize}
\end{lemma}

First we prove the lemmas, and then we prove the estimate $\Delta \lesssim n^{-8/7+\varepsilon''}$ using the lemmas. 

\begin{proof}[Proof of Lemma \ref{lem:single_step_prep_87}]
Over all axis parallel squares $Q \subset [0, 1]^2$ with side length $u\in [\Delta^{1/2}, u_0]$, choose one maximizing the expression 
$$
|P \cap Q| u^{-s}.
$$
By the choice of $Q$, for any square $Q' \subset [0, 1]^2$ with side $u' \in [\Delta^{1/2}, u_0]$ we have 
\begin{equation} \label{J}
    |P \cap Q'| (u')^{-s} \le |P_Q| u^{-s},
\end{equation}
i.e. $|P \cap Q'| \le (u')^s u^{-s} |P_Q|$. Note that $Q'$ has width $u'/u$ in the coordinate system of the square $Q$ and so this implies that $P_Q \subset Q$ is $s$-regular relative to $Q$ at scales above $\Delta^{1/2}/u$. 
On the other hand, by assumption, the set $P_Q$ contains no triangles of area less than $\Delta$, so in particular any $\Delta^{1/2} \times\Delta^{1/2}$ square contains at most $2$ points of $P_Q$. 
Thus, $P_Q$ is $s$-regular relative to $Q$ at all scales with constant~$2$. 

For any $v \in [u, u_0]$ and any square $Q'$ with side $v$ the inequality $v^{-s}| P\cap Q'| \le u^{-s} |P_Q|$ follows directly from definition. For $v \in [u_0, C u_0]$ we can divide $Q'$ into a union of $C^2$ squares with side $u_0$ and apply the inequality above to each one of them.
Similarly, by dividing $[0,1]^2$ into $u_{0}^{-2}$ axis parallel $u_{0} \times u_{0}$ squares we see that, by the pigeonhole principle, there exists an $u_{0} \times u_{0}$ square $Q_0 \subset [0, 1]^2$ which contains $\gtrsim u_0^2 |P|$ points of $P$. Taking $Q'=Q_0$ in \eqref{J}, we get $|P_Q| \gtrsim u^s u_0^{2-s} |P|$.

We conclude that the square $Q$ satisfies the properties of the lemma.
\end{proof}

\begin{proof}[Proof of Lemma \ref{lem:several_step_prep_87}]
    We construct the sequence of squares $Q_j$ by iterating Lemma \ref{lem:single_step_prep_87}. Namely, let $P^{(0)} = P$ and suppose that for some $j \ge 1$ we already defined the set $P^{(j-1)}$ and squares $Q_{1}, \ldots, Q_{j-1}$. 

    If $|P^{(j-1)}| < n/2$, then we stop the procedure. Otherwise, apply Lemma \ref{lem:single_step_prep_87} to the set of points $P^{(j-1)}$ and parameters $u_0, s$ as above. Then for some $u_j \in [\Delta^{1/2}, u_0]$ we obtain a $u_j \times u_j$ square $Q_j \subset [0, 1]^2$ and a set $P_{Q_j} = P^{(j-1)} \cap Q_j \subset P \cap Q_j$. Define $P^{(j)} = P^{(j-1)} \setminus P_{Q_j}$, and repeat this step with $j$ replaced by $j+1$.

    Suppose that this algorithm stopped after $t$ steps.
    To each index $j \in \{1, \ldots, t\}$ we can assign a pair of integers $c(j), d(j) \lesssim \log n$ such that $u_j \in [2^{-c(j)-1}, 2^{-c(j)}]$ and $|P_{Q_j}| \in [2^{d(j)}, 2^{d(j)+1}]$. Since the sets $P_{Q_j}$ cover at least $n/2$ points of $P$, we can find some $c, d$ such that the sets $P_{Q_j}$ with $c(j) = c$ and $d(j)=d$ cover $\gtrsim n/\log^2 n$ points of $P$. Let $S \subset \{1, \ldots, t\}$ be the set of these indices and let $u = 2^{-c-1}$ and $n_1 = 2^d$. Then the first condition of the lemma is satisfied for the family of squares $\{Q_j, ~j\in S\}$ and sets $P_{Q_j}$.

    The second condition follows from Lemma \ref{lem:single_step_prep_87}. Let $j_0$ be the minimum element of $S$, then we have $P' = \bigcup_{j\in S} P_{Q_j} \subset P^{(j_0)}$ and so the third condition also follows from Lemma \ref{lem:single_step_prep_87}. The lower bound $u^{-s} n_1 \gtrsim u_0^{2-s} n$ follows from the fact that $|P^{(j_0)}| > n/2$.
    
    Lastly, we would like to make sure that the squares $Q_j$ we constructed are pairwise disjoint. Currently this is not quite the case, but it turns out that a large subcollection of these squares does satisfy this property. Define a graph $G$ on the set $S$ where $i$ and $j$ are connected by an edge if squares $Q_i$ and $Q_j$ intersect. Note that for fixed $i \in S$ all squares $Q_j$ with $(i, j) \in G$ are contained in a square $\Upsilon$ with side $10 u$ around $Q_i$. Since $10u \le 10u_0 \le Cu_0$, by the third condition applied to $Q'=\Upsilon$, it follows that the union of sets $P_{Q_j}$ over $(i, j) \in G$ for fixed $i$ has size at most $\lesssim 10^s n_1$.

    So since sets $P_{Q_j}$ are pairwise disjoint and have size at least $n_1$, we conclude that the degree of $i$ in $G$ is bounded by some absolute constant $D \lesssim 10^s$. Since this holds for every vertex $i$ of $G$, the chromatic number of $G$ satisfies $\chi(G) \leq D+1$, and therefore there must exist an independent set $S' \subset S$ of size 
    $$|S'| \geq \frac{|S|}{\chi(G)} \geq \frac{|S|}{D+1}.$$ We thus obtain a family of squares $\mathcal Q = \{Q_j,~ j\in S'\}$ satisfying all conditions of the lemma.
\end{proof}

We are now ready to complete our alternative proof of the Koml\'os, Szemer\'edi, and Pintz bound.

\begin{proof}[Proof of $\Delta(n) \lesssim n^{-8/7+\varepsilon''}$]
    Fix a set of disjoint squares $\mathcal Q = \{Q_j\}$, sets $P_{Q_j} \subset P\cap Q_j$ and parameters $n_1, u$ as in Lemma \ref{lem:several_step_prep_87}. In particular, $P_{Q_j}$ is $s$-regular at all scales relative to $Q_j$ with constant $2$.
    Denote $\sigma = 1-\varepsilon$ and, for each square $Q_j$, apply Theorem \ref{marstrand} to the set $P_{Q_j}$ with parameters $s, \sigma$ and $C=2$. We get a set of lines $L_j$ spanned by pairs of points in $P_{Q_j}$ such that $|L_j| \ge K^{-1} |P_{Q_j}|^2$ and the set of directions $\theta(L_j) = \{\theta(\ell),~\ell \in L_j\}$ is $\sigma$-regular above scale $\delta \sim |P_{Q_j}|^{-1/s}$ with constant $K = K(2, s, \sigma)$. In particular, any interval $I \subset S^1$ of length $\delta$ contains at most $K \delta^\sigma |L_{j}|$ elements of $\theta(L_{j})$. 

    For each $j$, let $p_j \in Q_j$ be the middle point of the square $Q_j$ and note that all lines in $L_{j}$ are at distance at most $10 u$ to $p_j$. 
    Let $w = \max\{ C u,  n^{\varepsilon'} (nu)^{-1}\}$ where $C$ is an absolute constant consistent with the third condition from Lemma \ref{lem:several_step_prep_87}. 
    By design, we must have that any $w$-square contains $\lesssim (w/u)^{s}$ points of $\Pi = \{p_j\}$. 
   By Proposition \ref{prop:incidence_lower_bound_sep_weighted}, if  
    \begin{equation}\label{ineqmu}
    \mu_\Pi \mu_L^2 \le w^2 \eta n^{-\varepsilon}    
    \end{equation}
    holds with $\mu_\Pi = (w/u)^s / |\Pi|$, $\mu_L = K\delta^{\sigma}$ and $\eta = \delta$, then we have $B(w/10; \Pi, L) \gtrsim n^{-\varepsilon} \eta \mu_L^{-1}$, where $L = \bigcup L_j$ (in order to apply Proposition \ref{prop:incidence_lower_bound_sep_weighted}, we shift the lines in $L_j$ so that they pass through $p_j$, see Remark \ref{shift} after the statement of proposition).

    Using 
    $$
    \delta \sim n_1^{-1/s}, ~~ |\Pi| \gtrsim \frac{n}{n_1 \log^{2}n}, ~~ s = 1+\varepsilon, 
    $$
    one can easily check that (\ref{ineqmu}) indeed holds for any $w \gtrsim n^{\varepsilon'} (nu)^{-1}$. Thus, we get
    $$
    B(w/10; \Pi, L) \gtrsim n^{-\varepsilon} \eta \mu_L^{-1} \ge n^{-6\varepsilon}.
    $$
    Since each point $p_i \in \Pi$ has at least $n_1$ points of $P$ in its $u$-neighbourhood and $w \geq Cu$ for the $C$ above, we conclude that 
    $$
    B(w; P, L) = \frac{I(w; P, L)}{w |P| |L|} \ge \frac{n_1 I(w/10; \Pi, L)}{w |P| |L|} \gtrsim B(w/10;P', L) / \log^2 n \gtrsim n^{-7\varepsilon}.
    $$
    We are almost ready to apply Lemma \ref{lem:low_scales} to the set of points $P$, the set of lines $L$, $w = w$, $u = u$ and $\kappa = n^{-7\varepsilon}$. We have 
    $$
    I(w; P, L) \ge n^{-7\varepsilon} w |P| |L| \gtrsim |L|,
    $$
    and note that
    $$
    |L| = \sum_{j} |L_{j}| \gtrsim \frac{n}{n_1 (\log n)^{2}} n_1^2 = n n_1 \log^{-2}n.
    $$
    By Lemma \ref{lem:several_step_prep_87}, we have $u^{-s} n_1 \gtrsim u_0^{2-s} n$. So, using $u \ge \Delta^{1/2}$ and $u_0 = \Delta^{-1/2} n^{-1+\varepsilon'}$, we conclude that $n_1 \ge n^\varepsilon$ provided that $\varepsilon' \gtrsim \varepsilon$. It follows that $|L| \gtrsim |P|$ and so one can indeed apply Lemma \ref{lem:low_scales} to get
    \begin{equation}\label{2terms}
    \Delta/|\log \Delta| \lesssim \max\{u, u^{2/3} w^{1/3}\} |P|^{-1/3} |L|^{-1/3} \kappa^{-2/3} + u |P|^{-1} \kappa^{-1}.    
    \end{equation}
    Note that
    $$
    u|P|^{-1} \kappa^{-1} \le u_0 n^{-1+7\varepsilon} \le \Delta^{-1/2} n^{-2+2\varepsilon'}
    $$
    and so if this term dominates the right hand side of (\ref{2terms}) then we get $\Delta \le n^{-4/3 +\varepsilon''}$. Thus, we may assume that the first term is dominating. So after some simplifications we get
    $$
    \Delta \lesssim n^{\varepsilon'} \max\{u, u^{1/3} n^{-1/3}\} n^{-2/3} n_1^{-1/3},
    $$
    we have $n_1 \gtrsim u^s u_0^{2-s} n \gtrsim u \Delta^{-1/2}$ and $u \le u_0$ which gives
    $$
    \Delta^{5/6} \lesssim n^{\varepsilon'} \max\{u^{2/3}, n^{-1/3}\} n^{-2/3} \le u_0^{2/3} n^{-2/3+\varepsilon'} \le \Delta^{-1/3} n^{-4/3+2\varepsilon'}
    $$
    and so we get $\Delta \le n^{-8/7+\varepsilon''}$ provided that $\varepsilon' \lll \varepsilon''$. This completes the proof.
\end{proof}

\subsection{Proof of Theorem \ref{thm:main_theorem}}

In this section we finally prove our main result, namely that $\Delta(n) \lesssim n^{-8/7-1/2000}$ holds for sufficiently large $n$.

For large enough $n$, let $P \subset [0, 1]^2$ be an $n$-element set of points containing no triangles of area at most $\Delta$. Suppose that we already know that $\Delta(n) \lesssim n^{-\gamma_0+\varepsilon}$ for some constant $\gamma_0 \ge 8/7$ and all $\varepsilon >0$. To prove Theorem \ref{thm:main_theorem} it will be sufficient to take $\gamma_0 = 8/7$ but our argument works for an arbitrary starting exponent $\gamma_0$. By iterating this argument one can get slight improvements of the resulting bound.

Let 
\begin{itemize}
    \item $\varepsilon''' \ggg \varepsilon'' \ggg \varepsilon' \ggg \varepsilon > 0$ be a sequence of rapidly decreasing small constants,
    \item $u_0 > u_1$ be distance scale parameters, which will specified at the end of the proof,
    \item $s_1 = 1+\varepsilon$, $0.5+\varepsilon < s_2 < 1$ be dimension parameters (for upcoming applications of Theorem \ref{marstrand} and Theorem \ref{thm:finite_set_dir_set}).
\end{itemize}

The precise values of parameters $u_0, u_1, s_2$ which we need to take arise from an optimization problem which we get at the end of the argument. The following parameters will turn out to be sufficient to prove the bound in Theorem \ref{thm:main_theorem}:
\begin{equation}\label{choice}
\gamma_0 = 8/7, ~~ s_2 = 21/22, ~~ u_0 = n^{-3/7-0.0016}, ~~ u_1 = n^{-4/7+0.047}.    
\end{equation}
The argument below will be carried out with arbitrary values of parameters $\gamma_0, s_2, u_0, u_1$ and we will write down the inequalities which they need to satisfy along the way. Then, at the end, we collect all constraints and justify the choice (\ref{choice}).

\begin{lemma}[Single step preparation]\label{lem:single_step_prep_new}
Let $P$, $u_0, u_1$, and $s_1, s_2$ be arbitrary parameters as above. Then one of the following two options holds: 

(Square case) For some $u \in [u_1, u_0]$ there exists a $u \times u$ square $Q \subset [0,1]^2$ and $P_Q = P\cap Q$ such that 
\begin{itemize}
    \item $P_Q \subset Q$ is $s_1$-regular relative to $Q$ above scale  $\delta \sim n^{-\varepsilon'} u_1/u$,
    \item We have $v^{-s}|P\cap Q'| \lesssim u^{-s} |P_Q|$ for any square $Q'$ with side $v \in [u, C u_0]$. Here we can take $C=100$. In particular, $u^{-s} |P_Q| \gtrsim u_0^{2-s} |P|$. 
\end{itemize}

(Rectangle case) For some $a \le b \le u_1$, such that $ab \ge \Delta$, there exists an $a \times b$ rectangle $R \subset [0, 1]^2$ and $P_R = P\cap R$ such that 
\begin{itemize}
    \item $P_R \subset R$ is $s_2$-regular relative to $R$ at all scales with constant $2$, and for any strip $T$ of width $w \ge \Delta/ab$ relative to $R$ we have $|P_R\cap T| \lesssim w^\varepsilon |P_R|$, 
    \item We have $|P_R| \gtrsim n^{-\varepsilon'} b^{s_2} u_1^{s_1 - s_2} u_0^{2-s_1} |P|$, and any $w$-square $Q'$ satisfies 
    \begin{equation*}
    |P \cap Q'|n^{-\varepsilon'} \lesssim \begin{cases}
        (w/b)^{s_2} |P_R|, ~~~ w \in [b, u_1],\\
        (u_1/b)^{s_2} (w/u_1)^{s_1} |P_R|, ~~~ w \in [u_1, Cu_0],
    \end{cases}    
    \end{equation*}
    where we can take $C=100$.
\end{itemize}
\end{lemma}

\begin{lemma}[Several step preparation]\label{lem:several_step_prep_new}
Let $P$, $u_0, u_1$, and $s_1, s_2$ be as above. Then one of the following two options holds: 

(Square case)  There exists a set of disjoint squares $\mc Q = \{Q_j\}$, disjoint sets $P_{Q_j} \subset P\cap Q_j$, numbers $n_1\ge 1$ and $u \in [u_1, u_0]$ such that the following are satisfied. Denote $P' = \bigcup P_{Q_j}$.
\begin{itemize}
    \item $|P_{Q_j}| \in [n_1, 2n_1]$, $|Q_j| \in [u, 2u]$, and $|P'| \gtrsim \frac{n}{(\log n)^2}$.
    \item $P_{Q_j} \subset Q_j$ is $s_1$-regular relative to $Q_j$ above scale $\delta \sim n^{-\varepsilon'} u_1/u$, 
    \item We have $v^{-s_1}|P'\cap Q'| \lesssim u^{-s_1} n_1 $ for any square $Q'$ with side $v \in [u, Cu_0]$. Here we can take $C=100$. Furthermore, $u^{-s_1} n_1 \gtrsim u_0^{2-s_1} n$. 
\end{itemize}
(Rectangle case) There exists a set of pairwise disjoint rectangles $\mc R = \{R_j\}$, disjoint sets $P_{R_j} \subset P \cap R_j$, numbers $n_1\ge 1$ and $a \le b \le u_1$ such that $a b \ge \Delta$ and the following conditions are satisfied. Denote $P' = \bigcup P_{R_j}$.
\begin{itemize}
    \item $|P_{R_j}| \in [n_1, 2n_1]$, the sides of $R_j$ lie in the intervals $[a, 2a], [b, 2b]$ and $|P'| \gtrsim n^{1-2\varepsilon'}$,
    \item $P_{R_j} \subset R_j$ is $s_2$-regular relative to $R_j$ at all scales with constant $2$, and for any strip $T$ of width $w \ge \Delta/\operatorname{Area}(R_j)$ relative to $R$ we have $|P_{R_j}\cap T| \lesssim w^\varepsilon |P_{R_j}|$,
    \item $n_1 \gtrsim b^{s_2} u_1^{s_1 - s_2} u_0^{2-s_1} n^{1-\varepsilon'}$, and any $w$-square $Q'$ satisfies 
    \begin{equation*}
    |P' \cap Q'|n^{-\varepsilon'} \lesssim \begin{cases}
        (w/b)^{s_2} n_1, ~~~ w \in [b, u_1],\\
        (u_1/b)^{s_2} (w/u_1)^{s_1} n_1, ~~~ w \in [u_1, Cu_0],
    \end{cases}    
    \end{equation*}
    where we can take $C=100$.
\end{itemize}
\end{lemma}

\begin{proof}[Proof of Lemma \ref{lem:single_step_prep_new}]
Over all axis parallel squares $Q \subset [0, 1]^2$ with side length $u\in [u_1 n^{-\varepsilon'}, u_0]$, choose one maximizing the expression 
$$
|P \cap Q| u^{-s_1}.
$$
By the choice of $Q$, for any square $Q' \subset [0, 1]^2$ with side $u' \in [u_1 n^{-\varepsilon'}, u_0]$ we get $|P \cap Q'| \le (u')^{s_1} u^{-s_1} |P_Q|$. Like before, this implies that $P_Q \subset Q$ is $s_1$-regular relative to $Q$ at scales above $\delta = n^{-\varepsilon'} u_1/u $. Furthermore, by the pigeonhole principle, there exists an $u_{0} \times u_{0}$ axis-parallel square $Q_0 \subset [0, 1]^2$ which contains $\gtrsim u_0^2 |P|$ points of $P$. By taking $Q'=Q_0$, it follows from the maximality of $Q$ that $|P_Q| \gtrsim u^{s_1} u_0^{2-s_1} |P|$.

If we have $u \ge u_1$ then $Q$ satisfies the Square Case of the lemma. Otherwise, we have $u \in [u_1 n^{-\varepsilon'}, u_1)$. Among all rectangles $R \subset \R^2$ with sides $a \times b$ so that $a \le b \le u$ and $ab\ge \Delta$ choose one maximizing the expression
$$
|P \cap R| a^{-\varepsilon} b^{-s_2+\varepsilon}.
$$
Let $P_R = P \cap R$. Note that $R = Q$ is a possible choice of such a rectangle (provided that $u^2 \ge \Delta$ which will be true for our choice of $u_1$ and the lower bound $u \ge n^{-\varepsilon'} u_1$) and so we get
$$
|P_R| a^{-\varepsilon} b^{-s_2+\varepsilon} \ge |P_Q| u^{-s_2} \gtrsim u^{s_1-s_2} u_0^{2-s_1} |P| \ge n^{-\varepsilon'/2} u_1^{s_1-s_2} u_0^{2-s_1} |P|.
$$
Here we used the lower bound $u \ge n^{-\varepsilon'} u_1$ and an inequality $s_1-s_2 \le 0.5$.
We have $a \ge ab \ge \Delta$ and so $a^{-\varepsilon} \le \Delta^{-\varepsilon} \le n^{\varepsilon'/3}$ and we conclude that $|P_R| \ge n^{-\varepsilon'} b^{s_2} u_1^{s_1 - s_2} u_0^{2-s_1} |P|$. Similarly, for any $w$-square $Q'$ with $w \in [b,u]$ the definition of $R$ implies
$$
|P \cap Q'| \le w^{s_2} |P_R| a^{-\varepsilon} b^{-s_2+\varepsilon} \le n^{2\varepsilon} (w/b)^{s_2} |P_R|,
$$
and for $w \in [u, u_0]$ the definition of $Q$ implies
\begin{equation}\label{sqrbound}
|P \cap Q'| \le w^{s_1} |P_Q| u^{-s_1} \le w^{s_1} u^{s_2-s_1} a^{-\varepsilon} b^{-s_2+\varepsilon} |P_R|\le
\end{equation}
$$
\le w^{s_1} (n^{-\varepsilon'} u_1)^{s_2-s_1} n^{\varepsilon'/3} b^{-s_2} |P_R| \le n^{\varepsilon'} (u_1/b)^{s_2} (w/u_1)^{s_1} |P_R|,   
$$
where we again used the lower bound $u\ge n^{-\varepsilon'} u_1$ and the upper bound $a^{-\varepsilon} \le n^{\varepsilon'/3}$.

For squares $Q'$ with side $w \in [u_0, Cu_0]$ we can deduce the same bound as above by first tiling $Q'$ with $\sim C^2$ squares $u_0\times u_0$ and applying (\ref{sqrbound}) to each one of them (note that the implied constant in the bound gets multiplied by a power of $C$). We conclude that the rectangle $R$ satisfies the second conclusion of the Rectangle Case of the lemma. 

Now let $R' \subset R$ be a rectangle with sides $wa \times wb$ parallel to the sides of $R$. If $w^2 ab \ge \Delta$ then $R'$ is a valid choice for a rectangle in the definition of $R$ and so we get
$$
|P_R \cap R'| (aw)^{-\varepsilon} (bw)^{-s_2+\varepsilon}\le |P_R| a^{-\varepsilon} b^{-s_2+\varepsilon}
$$
$$
|P_R \cap R'| \le w^{s_2} |P_R|
$$
and so $P_R \subset R$ is $s_2$-regular relative to $R$ above scale $\sqrt{\Delta/ab}$. By assumption, the set $P_R$ contains no triangles of area at most $\Delta$ and so, in particular, any $(\Delta a/b)^{1/2} \times (\Delta b/a)^{1/2}$ rectangle $R' \subset R$ contains at most $2$ points of $P_R$. 
Thus, $P_R$ is $s_2$-regular relative to $R$ at all scales with constant $C=2$.

Finally, let $T$ be a strip of width $w$ relative to $R$. Then the intersection $R\cap T$ can be covered by a rectangle with sides $a'\le b'$ such that $a'b' \sim wab$ and $a'/b' \ge wa/b$. So as long as $wab \gtrsim \Delta$, we get 
$$
|P_R\cap T| (a')^{-\varepsilon} (b')^{-s_2+\varepsilon} \le |P\cap R'| (a')^{-\varepsilon} (b')^{-s_2+\varepsilon} \le |P_R| a^{-\varepsilon} b^{-s_2+\varepsilon},
$$
using a simple bound $(a'/a)^\varepsilon (b'/b)^{s_2-\varepsilon} \le (a'b'/ab)^\varepsilon \lesssim w^\varepsilon$ we obtain $|P_R\cap T| \lesssim w^\varepsilon |P_R|$. This shows that $R$ satisfies the first property of the Rectangle Case of the lemma.
\end{proof}

\begin{proof}[Proof of Lemma \ref{lem:several_step_prep_new}]
    We construct the sequence of squares and rectangles by iterating Lemma \ref{lem:single_step_prep_new}. Namely, let $P^{(0)} = P$ and suppose that for some $j \ge 1$ we already defined the set $P^{(j-1)}$ and for each $i = 1, \ldots, j-1$ we picked a square $Q_i$ or a rectangle $R_i$. 

    If $|P^{(j-1)}| < n/2$ then we stop the procedure. Otherwise, apply Lemma \ref{lem:single_step_prep_new} to the set of points $P^{(j-1)}$ and parameters $u_0, u_1, s_1, s_2$ as above. Then either a Square or Rectangle case holds. In the former case, we get some $u_j \in [u_1, u_0]$ and a $u_j\times u_j$ square $Q_j \subset [0, 1]^2$ and a set $P_{Q_j} = P^{(j-1)} \cap Q_j \subset P \cap Q_j$. Define $P^{(j)} = P^{(j-1)} \setminus P_{Q_j}$ and repeat this step with $j$ replaced by $j+1$. In the latter case, we get some $a_j \le b_j \le u_1$ such that $a_j b_j \ge \Delta$ and an $a\times b$ rectangle $R_j \subset [0, 1]^2$ and a set $P_{R_j} = P^{(j-1)} \cap R_j \subset P\cap R_j$. Define $P^{(j)} = P^{(j-1)} \setminus P_{R_j}$ and repeat this step with $j$ replaced by $j+1$.

    Suppose that this algorithm stopped after $t$ steps.
    To each index $j \in \{1, \ldots, t\}$ we can assign a letter $C_j \in \{Q, R\}$ depending on which of the two cases of Lemma \ref{lem:single_step_prep_new} happened at step $j$. For Square cases, we define a pair of integers $c(j), d(j) \lesssim \log n$ such that $u_j \in [2^{-c(j)-1}, 2^{-c(j)}]$ and $|P_{Q_j}| \in [2^{d(j)}, 2^{d(j)+1}]$. For Rectangle cases, we define integers $a(j), b(j), d(j)$ such that $a_j \in [2^{-a(j)-1}, 2^{-a(j)}]$,  $b_j \in [2^{-b(j)-1}, 2^{-b(j)}]$ and $|P_{R_j}| \in [2^{d(j)}, 2^{d(j)+1}]$.
    
    Since the sets $P_{Q_j}$ and $P_{R_j}$ together cover at least $n/2$ points of $P$, we can either
    \begin{itemize}
        \item find some $c_0, d_0$ such that the sets $P_{Q_j}$ with $C_j = Q$ and $c(j)=c_0$, $d(j) = d_0$ cover $\gtrsim n/\log^2 n$ points of $P$, or
        \item find some $a_0, b_0, d_0$ such that the sets $P_{R_j}$ with $C_j = R$ and $a(j)=a_0$, $b(j) = b_0$, $d(j) = d_0$ cover $\gtrsim n/\log^3 n$ points of $P$,
    \end{itemize}
    
    Consider the Square case first. Denote $S \subset \{1, \ldots, t\}$ the set of corresponding indices $j$. The first condition of Lemma \ref{lem:several_step_prep_new} is clearly satisfied with $n_1 = 2^{d_0}$, $u = 2^{-c_0-1}$, and note that $P' := \bigcup_{j\in S} P_{Q_j}$ satisfies $|P'| \gtrsim \frac{n}{(\log n)^2}$. 
    The second condition follows from the corresponding condition in Lemma \ref{lem:single_step_prep_new}. Let $j_0$ be the minimum element of $S$; by design $P' \subset P^{(j_0)}$, and so the third condition also follows from Lemma \ref{lem:single_step_prep_new}. The lower bound $u^{-s_1} n_1 \gtrsim u_0^{2-s} n$ follows from the fact that $|P^{(j_0)}| > n/2$.
    
Like in \S\ref{KPSsection}, we would also like to make sure that the squares $Q_j$ we ultimately select are pairwise disjoint. Define a graph $G$ on the set $S$ where $i$ and $j$ are connected by an edge if squares $Q_i$ and $Q_j$ intersect. Note that for fixed $i \in S$ all squares $Q_j$ with $(i, j) \in G$ are contained in a square with side $10 u$ around $Q_i$. So by the third condition, the union of sets $P_{Q_j}$ over $(i, j) \in G$ for fixed $i$ has size at most $\lesssim 10^{s_1} n_1$. So since sets $P_{Q_j}$ are pairwise disjoint and have size at least $n_1$, we conclude that the degree of $i$ in $G$ is bounded by some absolute constant $D \lesssim 10^{s_1}$. Like before, this means that the chromatic number of $G$ satisfies $\chi(G) \leq D+1$, and so there must exist an independent set $S' \subset S$ of size 
    $$|S'| \geq \frac{|S|}{\chi(G)} \geq \frac{|S|}{D+1}.$$
    The family of squares $\mathcal Q = \{Q_j,~ j\in S'\}$ satisfies all conditions of the Square case of the lemma. 

    Now consider the Rectangle case. Denote $S \subset \{1, \ldots, t\}$ the set of corresponding indices $j$. The first condition of Lemma \ref{lem:several_step_prep_new} is clearly satisfied with $n_1 = 2^{d_0}$, $a = 2^{-a_0-1}$, $b = 2^{-b_0-1}$ and a poly-logarithmic error instead of $n^{-2\varepsilon'}$. The second condition follows from the corresponding condition in Lemma \ref{lem:single_step_prep_new}. Let $j_0$ be the minimum element of $S$, then we have $P' = \bigcup_{j\in S} P_{Q_j} \subset P^{(j_0)}$ and so the third condition also follows from Lemma \ref{lem:single_step_prep_new}. The lower bound $ n_1 \gtrsim b^{s_2} u_1^{s_1-s_2} u_0^{2-s_1} n^{1-\varepsilon'}$ follows from the fact that $|P^{(j_0)}| > n/2$.

    Finally, we would like to make sure that in this case we are also able to select a large sub-family of rectangles which are pairwise disjoint. Define a graph $G$ on the set $S$ where $i$ and $j$ are connected by an edge if rectangles $R_i$ and $R_j$ intersect. Note that for fixed $i \in S$ all rectangles $R_j$ with $(i, j) \in G$ are contained in a square with side $10 b$ around $R_i$. So by the third condition, the union of sets $P_{R_j}$ over $(i, j) \in G$ for fixed $i$ has size at most $\lesssim n^{\varepsilon'} n_1$. So since sets $P_{R_j}$ are pairwise disjoint and have size at least $n_1$, we conclude that the degree of $i$ in $G$ is bounded by $Dn^{\varepsilon'}$ for some absolute constant $D$. This time the maximum degree of $G$ is not uniformly bounded, but by the same reasoning we are still able to select an independent set $S' \subset S$ of size $\ge |S|/Dn^{\varepsilon'}$. The corresponding family of rectangles $\mathcal R = \{R_j,~ j\in S'\}$ then satisfies all conditions of the Rectangle case of the lemma.
\end{proof}

The proof of Theorem \ref{thm:main_theorem} naturally splits into two parts depending on which of the cases of Lemma \ref{lem:several_step_prep_new} holds.

\begin{proof}[Square case:]  
    Suppose that $\mc Q$ is a family of squares satisfying the properties in the first case of Lemma \ref{lem:several_step_prep_new} with parameters $u, n_1$ and sets $P_{Q_j} \subset P\cap Q_j$. In particular, $P_{Q_j}$ is $s_1$-regular relative to $Q_j$ above scale $\delta \sim n^{-\varepsilon'} u_1/u$.

    Denote $\sigma = 1-\varepsilon$ and, for each square $Q_j$, apply Theorem \ref{marstrand} to the set $P_{Q_j}$ with parameters $s_1, \sigma$ and $C=1$. We get a set of lines $L_j$ spanned by pairs of points in $P_{Q_j}$ such that $|L_j| \ge K^{-1} |P_{Q_j}|^2$ and the set of directions $\theta(L_j) = \{\theta(\ell),~\ell \in L_j\} \subset S^1$ is $\sigma$-regular above scale $\delta$ with constant $K=K(1,s_1, \sigma)$. In particular, any interval $I \subset S^1$ of length $\delta$ contains at most $K \delta^\sigma |L_{j}|$ elements of $\theta(L_{j})$. 

    For each $j$, let $p_j \in Q_j$ be the middle point of the square $Q_j$ and note that all lines in $L_{j}$ are at distance at most $10 u$ to $p_j$. 
    Let $w = \max\{ C u, n^{\varepsilon'} \frac{n_1 u_1}{u^2 n}\}$ where $C$ is an absolute constant consistent with the third condition from Lemma \ref{lem:several_step_prep_new}. Note that any $w$-square contains $\lesssim (w/u)^s$ points of $\Pi = \{p_j\}$. 
    By Proposition \ref{prop:incidence_lower_bound_sep_weighted}, if  
    \begin{equation}\label{ineqmu2}
    \mu_\Pi \mu_L^2 \le w^2 \eta n^{-\varepsilon}    
    \end{equation}
    holds with $\mu_\Pi = (w/u)^s / |\Pi|$, $\mu_L = K\delta^{\sigma}$ and $\eta = \delta$, then we have $B(w/10; \Pi, L) \gtrsim n^{-\varepsilon} \eta \mu_L^{-1}$, where $L = \bigcup L_j$ (in order to apply Proposition \ref{prop:incidence_lower_bound_sep_weighted}, we shift again the lines in $L_j$ so that they pass through $p_j$, as in Remark \ref{shift}). Using 
    $$
    \delta \lesssim n^{-\varepsilon'} u_1/u, ~~ |\Pi| \gtrsim \frac{n}{n_1 \log^{2}n}, ~~ s = 1+\varepsilon, 
    $$
    one can easily check that (\ref{ineqmu2}) holds for any $w \gtrsim n^{\varepsilon'} \frac{n_1 u_1}{u^2 n}$.
    Thus, for our choice of $w$ we get
    $$
    B(w/10; \Pi, L) \gtrsim \eta \mu_L^{-1} \ge n^{-6\varepsilon}.
    $$
    Since each point $p_i \in P'$ has at least $n_1$ points of $P$ in its $u$-neighbourhood, we conclude that 
    $$
    B(w; P, L) = \frac{I(w; P, L)}{w |P| |L|} \ge \frac{n_1 I(w/10; \Pi, L)}{w |P| |L|} \gtrsim B(w/10;\Pi, L) / \log^2n \gtrsim n^{-7\varepsilon}.
    $$
    Apply Lemma \ref{lem:low_scales} to the set of points $P$, the set of lines $L$, $w = w$, $u = u$ and $\kappa = n^{-7\varepsilon}$. We have 
    $$
    I(w; P, L) \ge n^{-7\varepsilon} w |P| |L| \gtrsim |L|,
    $$
    $$
    |L| = \sum_{j} |L_{j}| \gtrsim n_1 n \log^{-2}n.
    $$
    By Lemma \ref{lem:several_step_prep_87}, we have $u^{-s} n_1 \gtrsim u_0^{2-s} n$. Using $u \ge u_1$ we conclude that $n_1 \ge n^\varepsilon$ provided that $\varepsilon' \ggg \varepsilon$ and $u_0 u_1 \gtrsim n^{\varepsilon'-1}$. 
    It follows that $|L| \gtrsim |P|$ and so one can apply Lemma \ref{lem:low_scales} and get:
    $$
    \Delta/|\log \Delta| \lesssim \max\{u, u^{2/3} w^{1/3}\} |P|^{-1/3} |L|^{-1/3} \kappa^{-2/3} + u |P|^{-1} \kappa^{-1}.
    $$
    If the second term dominates the right hand side then we get 
    \begin{equation*}\label{delta2nd}
    \Delta \lesssim n^{8\varepsilon} u |P|^{-1} \le u_0 n^{-1+8\varepsilon}.
    \end{equation*}
    Now suppose that the first term is dominating. Using $|L| \gtrsim n_1 n^{1-\varepsilon}$, $w = \max\{Cu, n^{\varepsilon'}\frac{n_1 u_1}{u^2 n}\}$ and some simplifications we get
    $$
    \Delta \lesssim n^{\varepsilon''} \max\{u, (n_1u_1/n)^{1/3}\} n^{-2/3} n_1^{-1/3} \lesssim n^{\varepsilon''} (u  n^{-2/3} n_1^{-1/3} + u_1^{1/3} n^{-1}),
    $$
    we have $n_1 \gtrsim u^s u_0^{2-s} n \gtrsim n^{-\varepsilon} u u_0 n$ and $u_1 \le u \le u_0$ and so we conclude
    \begin{equation}\label{deltaineq1}
    \Delta \lesssim  u_0^{1/3} n^{2\varepsilon''-1}.
    \end{equation}
    This completes the analysis of the first case. Note that the only condition needed to arrive at (\ref{deltaineq1}) was the assumption that $u_0 u_1 \gtrsim n^{\varepsilon'-1}$. In particular, note that this condition is satisfied for (\ref{choice}).
\end{proof}

\begin{proof}[Rectangle case:]
    Suppose that $\mc R$ is a family of rectangles satisfying the properties in the second case of Lemma \ref{lem:several_step_prep_new} with parameters $a \le b\le u_1$, 
 $n_1$ and sets $P_{R_j} \subset P\cap R_j$. In particular, $P_{R_j} \subset R_j$ is $s_2$-regular relative to $R_j$ at all scales with constant $2$ and for any strip $T$ of width $w \ge \Delta/\operatorname{Area}(R_j)$ relative to $R$ we have $|P_{R_j}\cap T| \lesssim w^\varepsilon |P_{R_j}|$.

    Denote $\sigma = s_2-\varepsilon$. 
    Let $w = c\tau(s, \sigma)^{1/\varepsilon} \gtrsim_{\varepsilon} 1$, where $\tau$ is the function from Theorem \ref{thm:finite_set_dir_set}. For an appropriate choice of a constant $c$, any strip $T$ of width $w$ relative to $R_j$ contains at most $\tau(s, \sigma) |P_{R_j}|$ points of $P_{R_j}$. 
    Let $\psi_j: [0, 1]^2 \rightarrow R_j$ be an affine isomorphism, then the set $\psi_j^{-1}(P_{R_j})$ satisfies the conditions of Theorem \ref{thm:finite_set_dir_set} with parameters $s_2, \sigma, w$ and $C=2$. So there is a set of lines $L'_j$ spanned by pairs of points in $\psi_j^{-1}(P_{R_j})$ such that $|L'_j| \ge K^{-1} |P_{R_j}|^2$ and the set of directions $\theta(L'_j) = \{\theta(\ell),~\ell \in L'_j\}\subset S^1$ is $\sigma$-regular above scale $\delta' \sim |P_{R_j}|^{-1/s_2}$ and with constant 
    $K = K(2,w, s_2, \sigma) \lesssim_\varepsilon 1$.

    In particular, any interval $I \subset S^1$ of length $\delta'$ contains at most $K (\delta')^\sigma |L'_{j}|$ elements of $\theta(L'_{j})$. 
    Define $L_j = \psi_j(L_j')$ and note that $L_j$ is a set of lines spanned by pairs of points in $P_{R_j}$. Since the rectangle $R_j$ has sides roughly $a \times b$ (up to a factor of 2), the map $\psi_j$ shrinks angles by a factor of at most $C b/a$. Thus, if $I \subset S^1$ is an interval of length $\lesssim a \delta' / b$ then its preimage $\psi^{-1}_j(I) \subset S^1$ has length at most $\delta'$. It follows that any interval $I\subset S^1$ of length $\delta \sim a \delta' / b$ contains at most $\mu_L |L_j|$ elements of $\theta(L_j)$ with $\mu_L = K (\delta')^\sigma$.
    
    For each $j$, let $p_j \in R_j$ be the middle point of the rectangle $R_j$ and note that all lines in $L_{j}$ are at distance at most $10 b$ to $p_j$. Let $\Pi = \{p_j\}$.
    Let $w \ge Cb$ and note that by Lemma \ref{lem:several_step_prep_new} any $w$-square $Q'$ satisfies 
    \begin{equation}\label{pprime}
    |\Pi \cap Q'| \lesssim M(w) := \begin{cases}
        (w/b)^{s_2}, ~~~ &Cb \le w \le u_1,\\
        (u_1/b)^{s_2} (w/u_1)^{s_1}, ~~~ &u_1 \le w \le u_0,\\
        (u_1/b)^{s_2} (u_0/u_1)^{s_1} (w/u_0)^2, ~~~ &w \ge u_0,
    \end{cases}    
    \end{equation}
    where the last row follows from the second by tiling $Q$ with $u_0\times u_0$ squares. 

    The following inequalities will be used extensively later on in the argument:
    \begin{equation}\label{abn}
        b \le u_1, ~~~ a \gtrsim \Delta n_1^{\gamma_0 -\varepsilon}/b, ~~~ n_1 \gtrsim b^{s_2} u_1^{s_1-s_2}u_0^{2-s_1} n^{1-\varepsilon'}.
    \end{equation}
    Recall that in Lemma \ref{lem:several_step_prep_new} we have $b \le u_1$. 
    By the assumption, we know that $\Delta(n_1) \lesssim n_1^{\varepsilon - \gamma_0}$ holds. Since $P_{R_j}$ is a set of size at least $n_1$ contained in a rectangle of area $ab$ and $P_{R_j}$ does not contain triangles of area at most $\Delta$, we thus have 
    $$
    ab n_1^{\varepsilon - \gamma_0} \gtrsim \Delta.
    $$
    We can use this to deduce a lower bound $a \gtrsim \Delta n_1^{\gamma_0-\varepsilon} / b$. By the lower bound from Lemma \ref{lem:several_step_prep_new} we have $n_1 \gtrsim b^{s_2}u_1^{s_1-s_2} u_0^{2-s_1} n^{1-\varepsilon'}$.
    
    By Proposition \ref{prop:incidence_lower_bound_sep_weighted} if the following inequality is  satisfied:
    \begin{equation}\label{ineqmu3}
        \mu_\Pi \mu_L^2 \le w^2 \eta n^{-\varepsilon}
    \end{equation}
    with $\mu_L = K (\delta')^{\sigma}$, $\eta = \delta \sim a\delta'/b$, $\delta' \sim n_1^{-1/s_2}$ and $\mu_\Pi = M(w)/|\Pi|$ then we have $B(w/10; \Pi, L) \gtrsim n^{-\varepsilon} \eta \mu_L^{-1}$, where $L = \bigcup L_j$.

    After substitution, (\ref{ineqmu3}) reduces to the following condition on $w$:
    \begin{equation}\label{wmess}
    n^{-1} n_1^{1/s_2 - 1} (b/a) M(w) \lesssim n^{-\varepsilon'} w^2.
    \end{equation}
    Note that the function $M(w) w^{-2}$ is increasing in $w$ and constant for $w \ge u_0$.
    Suppose that (\ref{wmess}) is not satisfied for $w=u_0$. We then get
    $$
    n^{\varepsilon'-1} n_1^{1/s_2-1} (b/a) \gtrsim u_0^{2} M(u_0)^{-1} = u_0^{2} (b/u_1)^{s_2} (u_1/u_0)^{s_1}
    $$
    $$
    a \lesssim n^{2\varepsilon'-1} b^{1-s_2} u_1^{s_2-1} u_0^{-1} n_1^{1/s_2-1}
    $$
    combining with the lower bound $a \gtrsim \Delta n_1^{\gamma_0-\varepsilon}/b$ (\ref{abn}) gives
    $$
    \Delta n_1^{\gamma_0-\varepsilon}/b \lesssim n^{2\varepsilon'-1} b^{1-s_2} u_1^{s_2-1} u_0^{-1} n_1^{1/s_2-1}
    $$
    $$
    \Delta n^{1-3\varepsilon'} \lesssim b^{2-s_2} u_1^{s_2-1} u_0^{-1} n_1^{1/s_2-1-\gamma_0}
    $$
    using the lower bound $n_1 \gtrsim b^{s_2}u_1^{s_1-s_2} u_0^{2-s_1} n$ and $s_1=1+\varepsilon$ we get
    $$
    \Delta n^{1-4\varepsilon'} \lesssim b^{2-s_2} u_1^{s_2-1} u_0^{-1} (b^{s_2}u_1^{1-s_2} u_0 n)^{1/s_2-1-\gamma_0}
    $$
    the exponent of $b$ above is equal to $3-2s_2 -s_2 \gamma_0$. Let us choose $s_2$ so that $3-2s_2 -s_2\gamma_0 \ge 0$ then we can use the upper bound $b \le u_1$ to conclude:
    \begin{equation}\label{delineq1}
    \Delta n^{1-4\varepsilon'} \lesssim u_1 u_0^{-1} (u_1 u_0 n)^{1/s_2-1-\gamma_0}    
    \end{equation}

    Now we consider the case when (\ref{wmess}) is satisfied for $w=u_0$. Let us pick the smallest $w \in [C b, u_0]$ which satisfies (\ref{wmess}). 
    Then (\ref{ineqmu3}) is satisfied for this choice of $w$ and we conclude that by Proposition \ref{prop:incidence_lower_bound_sep_weighted}
    $$
    B(w; \Pi, L) \gtrsim n^{-\varepsilon} \eta \mu_L^{-1} \sim n^{-\varepsilon} (a/b) n_1^{1-1/s_2}. 
    $$
    So since $|\Pi| \gtrsim n^{1-2\varepsilon'}/n_1$ and $b$-neighbourhood of any point in $P'$ contains at least $n_1$ points of $P$, we conclude that
    \begin{equation}\label{blower}
    B(w; P, L) \gtrsim n^{-3\varepsilon'} (a/b) n_1^{1-1/s_2}.    
    \end{equation}
    In order to apply Lemma \ref{lem:low_scales}, we need to check that $I(w;P, L) \ge C |L| \ge C^2 |P|$ for a large enough constant $C$. For the first inequality, we note that $ab/n_1 \ge \Delta$ since $P_{R_j}$ contains no triangles of area at most $\Delta$ and we have $b \le u_1$, $s_2 >0.5$, so
    \begin{equation*}\label{kappa}
    B(w; P, L) \gtrsim n^{-3\varepsilon'} a b^{-1} n_1^{1-1/s_2} \ge n^{-3\varepsilon'} \Delta b^{-2} \ge n^{-3\varepsilon'} \Delta u_1^{-2}, 
    \end{equation*}
    Suppose that $I(w; P, L) < C |L|$ then we get 
    $$
    \Delta u_1^{-2} < w^{-1} n^{-1+\varepsilon''} \le b^{-1} n^{-1+\varepsilon''} \le \Delta^{-1/2} n^{-1+\varepsilon''}
    $$
    by the assumption that $b \ge \Delta^{1/2}$. If $u_1 < n^{-3/8}$ then this inequality implies $\Delta \le n^{\varepsilon''' - 7/6}$, a much better bound than what we are aiming at. So as long as $u_1 < n^{-3/8}$ holds, we have the condition $I(w; P, L) \ge C |L|$.

    We have $|L| \sim n_1^2 |\Pi| \gtrsim n_1 n^{1-2\varepsilon'}$ and so to ensure that $|L| \ge C |P|$ it is enough to check that $n_1 \gtrsim n^{\varepsilon''}$. 
    By the last condition of Lemma \ref{lem:several_step_prep_new} we have
    $$
    n_1 \gtrsim b^{s_2} u_1^{s_1-s_2} u_0^{2-s_1} n \gtrsim n^{-2\varepsilon} b^{s_2} u_1^{1-s_2} u_0 n.
    $$
    If $n_1 \le n^{\varepsilon''}$ then using $b^2 \ge ab\ge \Delta$ this inequality gives an upper bound on $\Delta$:
    \begin{equation}\label{nsmall}
        \Delta^{s_2/2} \lesssim n^{2\varepsilon''} n^{-1} u_0^{-1} u_1^{s_2-1}.
    \end{equation}
    By optimizing parameters we will ensure that this bound leads to the desired bound on $\Delta$. So we may assume that $n_1\gtrsim n^{\varepsilon''}$ and in particular $|L| \ge C |P|$. Thus, we can apply Lemma \ref{lem:low_scales} to the set of points $P$ and the set of lines $L$ with $u = b$, $w = w$ and $\kappa = n^{-3\varepsilon'} (a/b) n_1^{1-1/s_2}$ (see (\ref{blower})). We conclude that
    \begin{equation}\label{deltaineq}
    \Delta/|\log \Delta| \lesssim \max\{b, b^{2/3} w^{1/3}\} |P|^{-1/3} |L|^{-1/3} \kappa^{-2/3} + b |P|^{-1} \kappa^{-1}.    
    \end{equation}
    
    First, let us check that the second term is negligible. Indeed, if it is not, then we obtain a very good upper bound on $\Delta$:
    $$
    \Delta/|\log \Delta| \lesssim b |P|^{-1} \kappa^{-1} \lesssim n^{-1} b^2 n_1^{1/s_2-1} / a \le n^{3\varepsilon'-1} b^2 n_1 /a
    $$
    using $n_1/a \le b/\Delta$ this gives $\Delta^2 \lesssim n^{4\varepsilon'-1} b^3 \le n^{4\varepsilon'-1} u_1^3$, which is a good enough bound, provided that, say, $u_1 \le n^{-4/9}$ (which is consistent with (\ref{choice})). Thus, we may focus on the case when the first term in (\ref{deltaineq}) is dominating. Plugging in the values for $|P|, |L| \sim n_1^2 |\Pi|, \kappa$ gives:
    $$
    \Delta \lesssim n^{\varepsilon''} \max\{b^{1/3}, w^{1/3}\} b^{4/3} a^{-2/3}  n^{-2/3} n_1^{\frac{2}{3s_2} -1}
    $$
    \begin{equation}\label{deltaineq4}
    \Delta^3 < n^{3\varepsilon''} w b^{4} a^{-2}  n^{-2} n_1^{\frac{2}{s_2} -3}     
    \end{equation}
    Recall that $w$ was defined as the smallest number greater than $Cb$ satisfying (\ref{wmess}). Define
    \begin{equation}\label{wstar}
    w_* = n^{2-4\varepsilon''} \Delta^3 a^2 b^{-4} n_1^{3-\frac{2}{s_2}},    
    \end{equation}
    then by (\ref{deltaineq4}) we have $w_* < w$ and so by the choice of $w$ we have either $w_* \le Cb$ or $w_*$ does not satisfy (\ref{wmess}). 

    If $w_* \le C u_1$ then
    $$
    w_*= n^{2-4\varepsilon''} \Delta^3 a^2 b^{-4} n_1^{3-\frac{2}{s_2}} \lesssim u_1,
    $$
    using $a \gtrsim \Delta n_1^{\gamma_0-\varepsilon}/b$ gives
    $$
    \Delta^5 n^{2-5\varepsilon''} \lesssim b^6 u_1 n_1^{\frac{2}{s_2}-3-2\gamma_0},
    $$
    using the lower bound $n_1 \gtrsim b^{s_2}u_1^{s_1-s_2} u_0^{2-s_1} n$ and $s_1=1+\varepsilon$ we get
    $$
    \Delta^5 n^{2-6\varepsilon''} \lesssim b^6 u_1 (b^{s_2}u_1^{s_1-s_2} u_0^{2-s_1} n)^{\frac{2}{s_2}-3-2\gamma_0}
    $$
    the exponent of $b$ above is equal to $8-3s_2-2s_2 \gamma_0 > 0$, so we can use the upper bound $b \le u_1$ to conclude:
    \begin{equation}\label{delineq3}
    \Delta^5 n^{2-7\varepsilon''} \lesssim u_1^7 (u_1 u_0 n)^{\frac{2}{s_2}-3-2\gamma_0}.
    \end{equation}    
    
    Otherwise, we have $w_* \ge Cu_1 \ge Cb$, so the only remaining case is when $w_* \in [Cu_1, u_0]$ does not satisfy (\ref{wmess}). That is, we have
    \begin{equation}\label{wstar2}
    n^{\varepsilon'-1} n_1^{1/s_2-1} (b/a) M(w_*) \gtrsim w_*^2    
    \end{equation}
    where the function $M$ is given by (\ref{pprime}).
    We have
    $$
    n^{1-\varepsilon'} n_1^{1-1/s_2} (a/b) \lesssim w_*^{-2} M(w_*) = w_*^{-2} (u_1/b)^{s_2} (w_* / u_1)^{s_1}
    $$
    $$
    a \lesssim   n^{2\varepsilon'-1} b^{1-s_2} u_1^{s_2-1} w_*^{-1} n_1^{1/s_2-1}
    $$
    we have $w_* = n^{2-4\varepsilon''} \Delta^3 a^2 b^{-4} n_1^{3-\frac{2}{s_2}}$ and so we get
    $$
    a^3 \Delta^3 \lesssim n^{5\varepsilon''-3} b^{5-s_2} u_1^{s_2-1} n_1^{\frac{3}{s_2}-4}
    $$
    combining with the lower bound $a \gtrsim \Delta n_1^{\gamma_0-\varepsilon}/b$ gives
    $$
    \Delta^6 n_1^{3\gamma_0-3\varepsilon} b^{-3} \lesssim n^{5\varepsilon''-3} b^{5-s_2} u_1^{s_2-1} n_1^{\frac{3}{s_2}-4}
    $$
    $$
    \Delta^6 n^{3-6\varepsilon''} \lesssim b^{8-s_2} u_1^{s_2-1} n_1^{\frac{3}{s_2}-4-3\gamma_0}
    $$
    using the lower bound $n_1 \gtrsim b^{s_2}u_1^{s_1-s_2} u_0^{2-s_1} n$ and $s_1=1+\varepsilon$ we get
    $$
    \Delta^6 n^{3-7\varepsilon''} \lesssim  b^{8-s_2} u_1^{s_2-1} (b^{s_2}u_1^{s_1-s_2} u_0^{2-s_1} n)^{\frac{3}{s_2}-4-3\gamma_0}
    $$
    the exponent of $b$ above is equal to $11-5s_2-3s_2 \gamma_0 > 0$, so we can use the upper bound $b \le u_1$ to conclude:
    \begin{equation}\label{delineq2}
    \Delta^6 n^{3-7\varepsilon''} \lesssim u_1^{7} (u_1 u_0 n)^{\frac{3}{s_2}-4-3\gamma_0}.
    \end{equation}
    
    We considered all possible cases and in each of them arrived at an upper bound on $\Delta$. It remains to choose parameters $u_1, u_0, s_2, \gamma_0$ and verify that the claimed bound on $\Delta$ follows.
\end{proof}

\begin{proof}[Choosing parameters]
    In the arguments above, we worked with unspecified parameters
    $$
    u_0, ~~ u_1, ~~ s_2, ~~\gamma_0
    $$
    for which we assumed the following properties:
    $$
    0.5 < s_2 < 1, ~~ s_2(2+\gamma_0) \le 3, ~~ u_1 < u_0, ~~ u_0 u_1 \gtrsim n^{\varepsilon' -1}, ~~ \Delta^{1/2} \lesssim u_1 \lesssim n^{-4/9}.
    $$
    Using these assumptions, we showed that one of the following bounds must hold (cf. (\ref{deltaineq1}), (\ref{delineq1}), (\ref{nsmall}), (\ref{delineq3}), (\ref{delineq2})): 
    \begin{align*}
        \Delta n^{1-2\varepsilon''} &\lesssim  u_0^{1/3},\\
        \Delta n^{1-4\varepsilon'} &\lesssim u_1 u_0^{-1} (u_1 u_0 n)^{1/s_2-1-\gamma_0},\\
        \Delta^{s_2/2} n^{1-2\varepsilon''} &\lesssim u_0^{-1} u_1^{s_2-1},\\
        \Delta^5 n^{2-7\varepsilon''} &\lesssim u_1^7 (u_1 u_0 n)^{\frac{2}{s_2}-3-2\gamma_0},\\
        \Delta^6 n^{3-7\varepsilon''} &\lesssim u_1^{7} (u_1 u_0 n)^{\frac{3}{s_2}-4-3\gamma_0}.
    \end{align*}
    Denote $\ell = \gamma_0 -1/s_2$.
    Denote $\Delta = n^{-\gamma_1 +\varepsilon'''}$ and $u_0 = n^{-\alpha}$, $u_1 = n^{-\beta}$, then after taking logs and using the fact that $\varepsilon'''$ is much larger than all other epsilons, the inequalities on $\Delta$ rewrite as follows:
    \begin{align*}
        3\gamma_1 &\ge 3+\alpha,\\
        \gamma_1 &\ge 2 - 2\alpha+ \ell (1-\alpha-\beta),\\
        \gamma_1 &\ge 2\beta + \frac{2}{s_2}(1-\alpha-\beta),\\
        5\gamma_1 &\ge 5 -3\alpha + 4\beta + 2\ell(1-\alpha-\beta),\\
        6\gamma_1 &\ge 7 -4\alpha + 3\beta + 3\ell(1-\alpha-\beta).
    \end{align*}
    where we can choose any $\alpha, \beta$ such that:
    $$
    0\le \alpha \le \beta, ~~ \alpha+\beta \le 1, ~~ 4/9 \le \beta \le \gamma_1/2
    $$
    and $s_2$ such that $0.5 < s_2 < 1$ and $s_2 \le \frac{3}{2+\gamma_0}$.

    Let $\gamma_0 = 8/7$, $s_2 = \frac{3}{2+\gamma_0} = \frac{21}{22}$  (so that $\ell = \frac{2}{21}$). For some $0 \le a \le b$, let $\alpha = \frac{3}{7}+a$ and $\beta = \frac{4}{7}-b$ (so that $1-\alpha-\beta = b-a$). Plugging these values into the system above gives:
    \begin{align*}
        3\gamma_1 &\ge 3 \cdot \frac{8}{7}+a,\\
        \gamma_1 &\ge \frac{8}{7} - 2a + \frac{2}{21}(b-a), \\
        \gamma_1 &\ge \frac{8}{7} - 2b + \frac{44}{21}(b-a),\\
        5\gamma_1 &\ge 6 - 3a - 4 b +\frac{4}{21}(b-a),\\
        6\gamma_1 &\ge 7 - 4a - 3 b + \frac{2}{7}(b-a).
    \end{align*}
    Observe that each of the first 3 inequalities above implies $\gamma_1 \ge 8/7+c$ if we take $a>0$ and $b$ significantly larger than $a$.
    On the other hand, the last two inequalities both imply a bound exceeding $8/7$ provided that both $a$ and $b$ are sufficiently small. So by choosing $1\ggg b \ggg a > 0$ we get a small improvement of the exponent. 
    
    Specifically, let us take $a = 0.0016$, $b = 0.047$. One can check that each of the inequalities above implies that $\gamma_1 \ge \frac{8}{7} + 0.0005 + c$ for some $c >0$. Taking $\varepsilon''' < c$ and all other epsilons correspondingly then implies the bound $\Delta < n^{-8/7-1/2000}$. 
    This completes the proof.
\end{proof}

\subsection{A further improved bound for homogeneous sets}

In this section we prove Theorem \ref{separated}. We recall that a homogeneous set $P \subset [0,1]^2$ is a well-spaced set, in the sense that there exists a constant $C > 0$ such that in every axis-parallel $n^{-1/2} \times n^{-1/2}$ square $Q$ contains at most $C$ points from $P$. We would like to show that for every $\varepsilon > 0$, there exist three points forming a triangle of area at most $$\Delta \leq n^{-7/6+\varepsilon}.$$

The proof will follow our approach from the previous sections, but we will now be able to use a much simpler network of pairwise disjoint squares, with easier access to the well-separated lines needed for the application of Proposition \ref{prop:incidence_lower_bound_sep_weighted}.   

\begin{proof}[Proof of Theorem \ref{separated}]
    Let $P \subset [0, 1]^2$ be as in the statement. Note that for $w \ge n^{-1/2}$, any $w\times w$ square $Q \subset [0, 1]^2$ can be covered by $\le 4 n w^2$ squares with side $n^{-1/2}$, so by assumption and a simple union bound we get that 
    \begin{equation} \label{hom}
    |P \cap Q| \le 4C n w^2.
    \end{equation}
    Let $u= n^{-1/2+2\varepsilon}$ and split the unit square $[0, 1]^2$ into $u\times u$ squares $Q_j$. Let $\mathcal Q$ be the set of squares $Q_j$ such that $|P \cap Q_j| \ge u^2 n/2$. Note that the squares from $\mathcal Q$ cover at least $n/2$ points of $P$. In particular, by (\ref{hom}) we have 
    $$
    u^{-2} \gtrsim \frac{n}{nu^2/2} \ge |\mathcal Q| \ge \frac{n/2}{4Cn u^2} \gtrsim u^{-2},
    $$ 
    i.e. $|\mathcal Q| \sim u^{-2}$.
    
    For each $Q_j \in \mathcal Q$, let $P_{Q_j} = P\cap Q_j$. By \eqref{hom}, note that the set $P_{Q_j}$ is $2$-regular relative to $Q_j$ at all scales with constant $8 C$. Indeed, for any $w \times w$ square $Q$, we have that
    $$|P_{Q_j} \cap Q| \leq 4Cnw^2 \leq 8 \cdot \frac{nu^2}{2} \cdot \left(\frac{w}{u}\right)^2 \leq 8C | P \cap Q_j| \left(\frac{w}{u}\right)^2.$$
    By Theorem \ref{marstrand} applied to $P_{Q_j}$ with $s=2$, $\sigma=1-\varepsilon$ and $C = 8C$, there is a set of lines $L_j$ spanned by pairs of points in $P_{Q_j}$ such that $|L_j| \ge |P_{Q_j}|^2/K$ and the set of directions $\theta(L_j) \subset S^1$ is $\sigma$-regular all scales above $\delta \sim |P_{Q_j}|^{-1/2}$ with constant $K = K(2, \sigma, 8C)$. In particular, every interval $I\subset S^1$ of length $\delta$ contains at most $K \delta^\sigma |L_j|$ directions from $\theta(L_j)$. 

    For each $Q_j \in \mathcal Q$ let $p_j$ be the middle point of the square $Q_j$ and note that all lines in $L_j$ are at distance at most $10u$ to $p_j$. 
    Let $w = 10 u$, and define $\Pi = \bigcup_{j: Q_{j} \in \mathcal{Q}} \{p_j\}$. Also, let $L = \bigcup_{j: Q_{j} \in \mathcal{Q}} L_j$. 
    
    By design, note that every $w$-square contains at most $100$ points of $\Pi$, so take $\mu_\Pi = 100/|\Pi| \sim u^{2}$, $\eta = \delta$, and $\mu_L = K \delta^{\sigma}$. We have
    $$
    \mu_\Pi \mu_L^2 \eta^{-1} \sim u^2 \delta^{2\sigma - 1} \lesssim u^2 \delta^{1/2} \le w^2 \delta^{1/2}.
    $$
    Since $\delta \sim |P_{Q_j}|^{-1/2}$ is upper bounded by a small negative power of $n$, this implies that Proposition \ref{prop:incidence_lower_bound_sep_weighted} is applicable to the set of points $\Pi$, the set of lines $L$, and scale $w$. Thus, we get $B(w; \Pi, L) \gtrsim n^{-\varepsilon} \eta \mu_L^{-1} \gtrsim n^{-2\varepsilon}$. Since every square $Q_j \in \mathcal Q$ contains $\ge u^2 n/2 \sim \frac{|P|}{|\Pi|}$ points of $P$, we get
    $$
    B(10w; P, L) \gtrsim B(w; \Pi, L) \gtrsim n^{-2\varepsilon}.
    $$
    Apply Lemma \ref{lem:low_scales} to the set $P$, the set of lines $L$ and $w=  10w = 100 u$, $\kappa = n^{-2\varepsilon}$. Then the smallest area $\Delta$ of a triangle in $P$ satisfies:
    $$
    \Delta / |\log \Delta| \lesssim u |P|^{-1/3} |L|^{-1/3} \kappa^{-2/3} + u |P|^{-1} \kappa^{-1},
    $$
    we have $|P| = n$, $u = n^{-1/2+2\varepsilon}$ and $|L| \gtrsim n^2 u^2 \sim n^{1+4\varepsilon}$, so after simplifications we get $\Delta \lesssim n^{-7/6 + 5\varepsilon}$, as desired.
\end{proof}

Note that $u \approx n^{-1/2}$ is the smallest scale at which we can ensure pairs of points of distance $\leq u$. If $L$ denotes the set of lines spanned by the pairs of points in $P$ which lie at distance at most $u$, then $|L| \approx n$, and recall from the discussion in \S\ref{sec:inductive_step_high_low} that the smallest scale that can be reached by the high-low method is $(|P||L|)^{-1/3} \approx n^{-2/3}$. Here and throughout this final discussion $\approx$ means equality up to small powers of $n^{-\varepsilon}$. Theorem \ref{separated} shows that this benchmark can be achieved for the Heilbronn triangle problem for sets of points $P \subset [0,1]^{2}$ which are homogeneous. It is also worth contrasting this story with Schmidt’s method, which we outlined in \S\ref{sec:Incidences}. The Riesz energy estimate from \eqref{Riesz} gives a better bound when the points are more concentrated (and the Riesz $2$-energy has a logarithmic divergence even if the points are maximally spread out), while Roth’s method and our approach are on the other hand most effective when the point set is well-distributed in the unit square (as showcased by Theorem \ref{separated}). We would like to add that it is also conceivable that for the points $P$ and lines $L$ involved in the Heilbronn triangle problem analysis, one could continue to have a nontrivial number of incidences all the way down to scale $\approx n^{-1}$, rather than $\approx n^{-2/3}$. This would lead to an estimate for $\Delta$ of the form $\Delta \lessapprox n^{-3/2}$, which even for the case when $P$ is homogeneous in $[0,1]^2$ currently seems out of reach to us (as it would have to bypass the Szemer\'edi-Trotter obstruction). For the reader's convenience, we include a table below collecting all these different thresholds and all the previous best known bounds for the Heilbronn triangle problem, in chronological order.

\begin{table}[h]
\centering
\begin{tabular}{l | l l}
	\textbf{Authors} &  \\ \hline \\ [-1.8ex]
	Trivial bound & $\Delta \lesssim n^{-1}$ \\ 
	Roth \cite{Roth1} (1951) & $\Delta \lesssim n^{-1} (\log \log n)^{-1/2}$ \\ 
	Schmidt \cite{Schmidt} (1972) & $\Delta \lesssim n^{-1} (\log n)^{-1/2}$ \\ \hline\\[-1.8ex]
	Roth \cite{Roth2,Roth3,Roth4} (1972-73) & $\Delta \lesssim n^{-1-\mu}$ & $\sim n^{-1.117\ldots}$ \\ 
	KPS \cite{KPS81} (1981) & $\Delta \lesssim n^{-1-1/7}$ & $\sim n^{-1.142\ldots} $ \\ 
	This paper (2023) & $\Delta \lesssim n^{-1-1/7-1/2000}$ & $\sim n^{-1.143 \ldots}$ \\ \hline\\ [-1.8ex]
	Limit of high-low & $\Delta \lesssim n^{-7/6+\varepsilon}$ & $\sim n^{-1.167\ldots}$ \\ 
	Limit of incidence setup & $\Delta \lesssim n^{-3/2+\varepsilon}$ & $\phantom{\sim}\ n^{-1.5+\varepsilon}$ \\ 
	Best possible & $\Delta \lesssim n^{-2+\varepsilon}$ & $\phantom{\sim}\ n^{-2+\varepsilon}$
\end{tabular}
\vspace{0.1in}
\label{table:bound_table}
\end{table}

\section{Final remarks}

In this paper we showed that in every set of $n$ points chosen inside a unit square there exists a triangle of area less than $n^{-8/7-1/2000}$, thereby improving upon the classical result of Koml\'os, Pintz and Szemer\'edi from \cite{KPS81} by a polynomial factor. We also gave a further improved bound of $n^{-7/6}$, in the case when $P$ is homogeneous. Perhaps most importantly, our approach established new connections between the Heilbronn triangle and various themes in incidence geometry and projection theory that are closely related to the discretized sum-product theorem of Bourgain \cite{Bourgain03}. 

Apart from the problem of further improving the bounds for the Heilbronn triangle problem, there are several other related open questions which unfortunately are a bit less known. We would like to use this section as an opportunity to draw attention towards a few of our favorites. We hope that the ideas introduced in this paper will also lead to progress on some of these. 

The first one is a problem due to Motzkin and Schmidt (cf. \cite{BMP} and \cite{Beck90}). For every set of points $P \subset [0,1]^2$, let $\Phi(P)$ denote the smallest real number $w \geq 0$ for which there is a strip of width $w$ containing at least three points from $P$. Let
$$\Phi(n):=\max \Phi(P),$$
where the maximum is taken over all sets of $n$ points $P \subset [0,1]^2$. 

\begin{conj} \label{MS}
$$\Phi(n)=o(1/n).$$
\end{conj}

A simple pigeonholing argument shows that $\Phi(n) \leq 3/n$. Using an analytic technique in the style of Roth's method discussed in \S\ref{sec:Incidences}, Beck \cite{Beck90} proved this conjecture in the
special case when the set is uniformly distributed in the sense that it contains
precisely one point from each $n^{-1/2} \times n^{-1/2}$ square (note that this is a stronger condition than being homogeneous). No nontrivial upper bound is known in general. 

It is worth emphasizing that Conjecture \ref{MS} is directly related to the Heilbronn triangle problem. Clearly,
$$\Delta(n) \leq \Phi(n),$$
but the inequality goes only one way: a triangle of area $o(1/n)$ is not necessarily be contained in a strip of width $o(1/n)$. 

The next problem is a natural generalization of the Heilbronn triangle problem, which was also proposed by Schmidt in \cite{Schmidt}. Let $\Delta_{k}(n)$ be the smallest number $\Delta = \Delta(n)$ such that in every configuration of $n$ points in the unit square $[0,1]^2$ one can always find $k$ points among them whose convex hull has area at most $\Delta_{k}(n)$. 

\begin{conj} \label{MS}
For every $k \geq 3$,
$$\Delta_{k}(n)=o(1/n).$$
\end{conj}

Clearly, $\Delta(n)=\Delta_{3}(n) \leq \Delta_{4}(n) \leq \ldots \leq \Delta_{k}(n)$ holds for al $k \geq 3$, so Conjecture \ref{MS} holds for $k=3$. However, rather surprisingly, the question remains open for every $k \geq 4$. The best known lower bound is due to Leffman \cite{Leffman} and comes from a generalization of the semi-random construction from \cite{KPS81}: for every $k \geq 3$,
\begin{equation}\label{Leff}
    \Delta_{k}(n) \gtrsim \frac{(\log n)^{\frac{1}{k-2}}}{n^{1+\frac{1}{k-2}} }.
\end{equation}

In the same spirit, one can also talk about an analogous generalization of the Motzkin-Schmidt problem, namely the question for strips that contain at least $k$ points from $P$, where $k \geq 3$ is fixed integer. Like before, let $\Phi_{k}(P)$ be the smallest $w \geq 0$ for which there is a strip of width $w$ containing at least $k$ points from $P$, and define $\Phi_{k}(n):=\max \Phi_{k}(P)$, where the maximum is taken again over all configurations of $n$ points $P \subset [0,1]^2$. Naturally, we have that $\Delta_{k}(n) \leq \Phi_{k}(n)$ for all $k \geq 3$, which already due to \eqref{Leff} suggests a rather interesting behavior for $\Phi_{k}(n)$ when $k$ becomes large. 

We would like to end with a question of Bourgain in the opposite direction.

\begin{question} 
Is it possible to find a set of $n$ points $P \subset [0,1]^2$ and an absolute constant $C > 0$ such that any tube of width $1/n$ contains at most $C$ points from $P$?
\end{question}

As Varj\'u writes in [Remembering Jean Bourgain (1954-2018), AMS Notices June 2021, p. 957], the motivation for this problem comes from a possible construction of spherical harmonics as a combination of Gaussian beams, which would have $L^{\infty}$ norm bounded by a constant independently of the degree. If one could show that no such constant existed, this would immediately imply that $\Phi_{k}(n) = o(1/n)$ holds for every $k \geq 3$. 

\bigskip

\bigskip

{\bf{Acknowledgements}}. We would like to thank Larry Guth and Peter Sarnak for inspiring discussions.

\bigskip

\appendix

\section{Proof of direction set estimates}\label{app:proof_of_dir_set_estimates}

In this section we prove Theorem \ref{thm:finite_set_dir_set} from the version that is stated in \cite{OSW}. 

\begin{theorem}[{\cite[Corollary 2.18]{OSW}}]\label{thm:OSW_original_thm}
For all $0 < \sigma < s \leq 1$ and $C, \varepsilon, w > 0$, there exists $\tau = \tau(\varepsilon, \sigma, s) > 0$ and $K = K(C, w, \varepsilon, s, \sigma) > 0$ such that the following holds. Let $\mu, \nu$ be probability measures on $[0,1]^2$, and let $X := \operatorname{supp}(\mu)$ and $Y := \operatorname{supp}(\nu)$. Suppose that
\begin{align}
    \mu(B(x, r)) &\leq Cr^s\quad\text{and}\quad \nu(B(x, r)) \leq C r^s\qquad \text{for all $x \in [0,1]^2$ and $r > 0$}, \label{eq:measures_frostman_OSW}\\ 
    \dist(X, Y) &\geq C^{-1}, \label{eq:measures_sep_OSW}
\end{align}
and 
\begin{equation}
    \max\{\mu(T), \nu(T)\} \leq \tau\label{eq:OSW_thin_on_tubes}
\end{equation}
for all $w$-tubes $T \subset \R^2$. Then there exists a Borel set $G \subset X\times Y$ with $(\mu\times \nu)(G) \geq 1-\varepsilon$ such that for any $x \in X$, 
\begin{equation*}
    \nu(T\cap G|_x)\leq K\cdot r^{\sigma}\quad \text{for all $r > 0$ and all $r$-tubes $T$ containing $x$}.
\end{equation*}
Here, $G|_x = \{y \in Y\, :\, (x, y) \in G\}$ is the restriction. We say that $G$ witnesses thin tubes. 
\end{theorem}
To prove Theorem \ref{thm:finite_set_dir_set} from Theorem \ref{thm:OSW_original_thm} we have to convert between measures satisfying the Frostman condition at all scales and a discrete set of points satisfying the Frostman condition down to scale $\delta$. Theorem \ref{thm:OSW_original_thm} estimates the dimension of radial projections of $Y$ onto points in $X$, but this easily implies a direction set estimate. In what follows $\Area(\cdot)$ is the Lebesgue measure. 
\begin{proof}[Proof of Theorem \ref{thm:finite_set_dir_set} from Theorem \ref{thm:OSW_original_thm}]
Let $0 < s < 1$, $\sigma < s$, and $C, w > 0$ be given. Let $\tau = \tau(\varepsilon, \sigma, s) > 0$ be as in Theorem \ref{thm:OSW_original_thm}. 
Let $P \subset [0,1]^2$ be a set of points such that 
\begin{equation}
    |P\cap Q| \leq C|Q|^s\, |P|,\quad \text{for any square $Q$ with $|Q| > \delta$},\label{eq:s_reg_pf_OSW}
\end{equation}
and suppose that 
\begin{equation*}
    |P\cap T| \leq \tau |P|
\end{equation*}
for all $w$-tubes $T$, $w > 2\delta$. 

By (\ref{eq:s_reg_pf_OSW}), there exist two $1/20C$-squares $Q_1$, $Q_2$ that each have $\geq (1/400C^2) |P|$ points of $P$ and which are $1/20C$-separated. Let $P_1 = P \cap Q_1$ and $P_2 = P \cap Q_2$, and set 
\begin{equation*}
    \mu = \frac{1}{|P_1|} \sum_{p \in P_1} \frac{1}{\Area(B_p)} 1_{B_p},\quad \nu = \frac{1}{|P_2|} \sum_{p\in P_2} \frac{1}{\Area(B_p)} 1_{B_p},\quad B_p = B(p, \delta/100).
\end{equation*}
Let $X = \supp(\mu)$, $Y = \supp(\nu)$. 
Notice that $\mu, \nu$ satisfy the Frostman condition (\ref{eq:measures_frostman_OSW}) with a constant $\tilde C \lesssim C$, they are $1/\tilde C$ separated, and they satisfy (\ref{eq:OSW_thin_on_tubes}) at scale $w/2$.

We apply theorem Theorem \ref{thm:OSW_original_thm} with $\varepsilon = 1/2$. Let $G \subset X \times Y$ be the Borel set witnessing thin tubes, with $(\mu\times\nu)(G) \geq 1/2$.
Let 
\begin{align*}
    \tilde G &= \Bigl\{(p_1, p_2)\in P_1 \times P_2\, :\, \Area(G\cap (B_{p_1} \times B_{p_2})) \geq \frac{1}{4}\Area(B_{p_1})\Area(B_{p_2})\Bigr\}, \\ 
    L &= \{\ell_{p_1, p_2}\, :\, (p_1, p_2) \in \tilde G\}. 
\end{align*}
We have 
\begin{align*}
    (\mu \times \nu)(G)  &= \frac{1}{|P_1|\,|P_2|}\sum_{(p_1, p_2) \in P_1 \times P_2} \frac{1}{\Area(B_{p_1})\Area(B_{p_2})}\Area(G\cap (B_{p_1} \times B_{p_2})) \\ 
    &\leq \frac{|L|}{|P_1|\, |P_2|} + \frac{1}{4}.
\end{align*}
Using the lower bound $(\mu\times \nu)(G) \geq 1/2$, we get $|L| \geq \frac{1}{4}|P_1|\, |P_2| \gtrsim |P|^2$. 

Let $p_1 \in P_1$ and let $T$ be an $r$-tube containing $p_1$ with $r > \delta$. We prove an estimate for $| \tilde G|_{p_1} \cap T |$. Let $2T$ be the $2r$-tube with the same central line.
For any $x \in [0,1]^2$, we have $\nu(2T\cap G|_x) \lesssim Kr^{\sigma}$. Integrating, 
\begin{equation*}
    \frac{1}{\Area(B_{p_1})}\int_{x \in B_{p_1}} \nu(2T \cap G|_{x})\, dx \lesssim Kr^{\sigma}
\end{equation*}
as well. We expand this integral using Fubini's theorem,
\begin{align*}
    \frac{1}{\Area(B_{p_1})}\int_{x \in B_{p_1}} \nu(2T \cap G|_{x})\, dx &= \frac{1}{\Area(B_{p_1})} \int_{x \in B_{p_1}}  \frac{1}{|P_2|}\sum_{p_2 \in P_2}\frac{1}{\Area(B_{p_2})}\Area(B_{p_2}\cap 2T\cap G|_x)\, dx \\ 
    &\geq \frac{1}{\Area(B_{p_1})} \int_{x \in B_{p_1}}  \frac{1}{|P_2|}\sum_{p_2 \in P_2\cap T}\frac{1}{\Area(B_{p_2})}\Area(B_{p_2}\cap G|_x)\, dx \\ 
    &=\frac{1}{\Area(B_{p_1})\Area(B_{p_2})|P_2|} \sum_{p_2 \in P_2 \cap T} \Area(G\cap (B_{p_1}\times B_{p_2})) \\ 
    &\geq \frac{1}{4|P_2|}\, | \tilde G|_{p_1} \cap T |,
\end{align*}
so
\begin{equation}\label{eq:point_set_rad_proj}
    |\tilde G|_{p_1} \cap T| \lesssim K|P_2|\, r^{\sigma}. 
\end{equation}
One should read (\ref{eq:point_set_rad_proj}) as a radial projection estimate. Now we prove a direction set estimate. Let $I \subset S^1$ be an interval with $|I| \geq \delta$. For any $p_1 \in P_1$, the locus of points $p_2$ such that $\theta(p_1, p_2) \in I$ is a double cone with angle $|I|$, which is contained in the width $20|I|$ tube $T_{p_1}$ centered at $p_1$ and pointing in a direction $\theta_I \in I$. Thus
\begin{equation*}
    \#\{\ell \in L\, :\, \theta(\ell) \in I\} \leq \sup_{p_1 \in P_1} |\tilde G|_{p_1}\cap T_{p_1}| \leq 4K|P_1||P_2|\,|I|^{\sigma} \lesssim K|L|\, |I|^{\sigma}. 
\end{equation*}
\end{proof}

Finally, we show how Theorem \ref{marstrand} formally follows from Theorem \ref{thm:finite_set_dir_set}, although we emphasize that Theorem \ref{marstrand} can be proved more easily directly (there is a proof using the high-low method). 

\begin{proof}[Proof of Theorem \ref{marstrand} from Theorem \ref{thm:finite_set_dir_set}]
Let $s > 1$. If $P \subset [0,1]^2$ is $s$-regular above scale $\delta$ with constant $C$, then any $w \times 1$ tube $T$ with $w > \delta$ can be covered with at most $2w^{-1}$ squares of side length $w$. Applying the Frostman regularity condition to these squares and summing, we get
\begin{equation*}
    |P\cap T| \leq 2C\, w^{-1}\, (|P|w^s) \leq 2C|P|\, w^{s-1} \lll |P|
\end{equation*}
so the conditions of Theorem \ref{thm:finite_set_dir_set} are easily satisfied. 
\end{proof}

\end{document}